\documentclass[11pt,reqno]{amsart}
\usepackage[utf8]{inputenc}
\usepackage{amsmath}
\usepackage{amsfonts}
\usepackage{amssymb}
\usepackage{amsthm}
\usepackage{appendix}
\usepackage{indentfirst}
\usepackage{mathtools}
\mathtoolsset{showonlyrefs=true}
\usepackage{enumitem}
\usepackage{hyperref}
\usepackage{geometry}
\usepackage{xcolor}

\hypersetup{
colorlinks=true,
linkcolor=blue,
citecolor=blue,
}

\title[Stability of elliptic Hartree two-solitons]{Finite-time stability of two-soliton solutions of the Hartree equation with elliptic trajectories}
\author{Yutong Wu}
\address{Department of Mathematics, Yale University, New Haven, CT 06511} 
\email{yutong.wu.yw894@yale.edu}

	

\begin{document}

\numberwithin{equation}{section} 

\newtheorem{thm}{Theorem}
\newtheorem{lem}{Lemma}[section]
\newtheorem{prop}[lem]{Proposition}
\newtheorem{defn}[lem]{Definition}
\newtheorem{rmk}[lem]{Remark}
\newtheorem{cor}{Corollary}
\newtheorem{ass}{Assumption}

\def\ol{\overline}
\renewcommand{\Re}{\mathrm{Re}}
\renewcommand{\Im}{\mathrm{Im}}
\def\eps{\varepsilon}

\def \N{\mathbb{N}}
\def \R{\mathbb{R}}
\def \C{\mathbb{C}}

\def \d{\mathrm{d}}
\def \dt{\mathrm{d}t}
\def \dx{\mathrm{d}x}

\def \E{\mathcal{E}}
\def \G{\mathcal{G}}
\def \H{\mathcal{H}}
\def \L{\mathcal{L}}

\newcommand{\f}[2]{\frac{#1}{#2}}

\begin{abstract}
We prove the finite-time stability of two-soliton solutions of the three-dimensional gravitational Hartree equation whose centers remain close to an elliptic solution of the Kepler two-body problem.  
\end{abstract}

\maketitle

\section{Introduction} \label{sec intro}

We consider the Hartree equation in three dimensions
\begin{equation} \label{eq hartree}
    iu_t+ \Delta u- \phi_{|u|^2}u=0,
\end{equation}
where $u: \R \times \R^3 \to \C$ and $\phi_{|u|^2}= \Delta^{-1}(|u|^2)= -\frac{1}{4\pi|x|} *|u|^2$. The equation arises in physics as an effective model in the mean-field limit of many-body quantum systems.

We review some properties of the Hartree equation. We refer to the monograph \cite{cazenave2003semilinear} for general background on dispersive equations, much of which applies to the Hartree equation. 

The Hartree equation is mass ($L^2$)-subcritical, so the Cauchy problem of \eqref{eq hartree} is globally wellposed in $H^1$. 

The equation also possesses a large family of symmetries. Namely, if $u$ solves \eqref{eq hartree}, then for any $(t_0, \alpha_0, \beta_0, \lambda_0, \gamma_0) \in \R \times \R^3 \times \R^3 \times \R_+ \times (\R/2\pi \mathbb{Z})$,
\begin{equation} \label{eq symmetry}
    v(t,x)= \lambda_0^2 u(\lambda_0^2 t+t_0, \lambda_0 x- \alpha_0- \beta_0 t) e^{i (\frac{1}{2} \beta_0 \cdot x- \frac{1}{4}|\beta_0|^2 t+ \gamma_0 )}
\end{equation}
also solves \eqref{eq hartree}. Here $\alpha_0$ corresponds to translation, $\beta_0$ corresponds to Galilean transformation, $\lambda_0$ corresponds to scaling, and $\gamma_0$ corresponds to phase rotation. Noether's theorem yields the following conserved quantities:
\begin{align}
    &\text{Mass: } && \mathcal{M}(u)= \int |u(t,x)|^2 \d x, \\
    &\text{Momentum: } && \mathcal{P}(u)= \int \Im \big( \nabla u(t,x) \overline{u(t,x)} \big) \d x, \\
    &\text{Hamiltonian: } && \H(u)= \frac{1}{2} \int |\nabla u(t,x)|^2 \d x- \frac{1}{4} \int \big| \nabla \phi_{|u|^2} (t,x) \big|^2 \d x.
\end{align}

There is a special type of solutions of \eqref{eq hartree}, called \emph{solitons}. These are solutions $u$ of the form $u(t,x)= e^{it} W(x)$, where $W$ satisfies $\Delta W- \phi_{|W|^2} W= W$. We denote by $Q$ the unique positive radial solution to 
\begin{equation} \label{eq ground state}
    \Delta Q- \phi_{Q^2}Q= Q.
\end{equation}
For the existence and uniqueness of $Q$, see \cite{Liebgroundstate}. The function $Q$ is called the \emph{ground state}. It decays exponentially, so it can be viewed as a well-localized object. The ground state can also be characterized in a variational way as the minimizer of $\H$ with a fixed $L^2$ norm. The orbital stability of the ground state solitary wave follows by combining \cite{orbitalstability} and \cite{Liebgroundstate}.

Using the symmetries of the equation, the ground state soliton $e^{it} Q(x)$ generates a family of solitary waves. \emph{Multisolitons}, or \emph{Multisolitary waves}, are sums of solitary waves and are expected to play an essential role in the study of long-time dynamics. For general dispersive equations, the \emph{soliton resolution conjecture} claims that a generic solution splits into a multisolitary wave and a radiation term as time tends to infinity. For local dispersive equations, there is extensive literature on this subject; for example, we refer to \cite{DKMwaveodd, JendreyLawrieNLW} and references therein. Multisoliton solutions have also been extensively studied in different local models; see, for example, \cite{KsolitarywavesofNLS, MartelgKdV, NsolitonKleinGordon} for existence results and \cite{MMTstabilityforNLS, MMTstabilityforgKdV} for stability results.

The long-time dynamics of the Hartree equation is more subtle than that of local models such as the nonlinear Schr\"odinger equation (NLS), due to the nonlocality of the nonlinearity. Although $Q$ decays exponentially, the nonlinear potential $\phi_{Q^2}$ decays only at a polynomial rate and is long-range:
\begin{equation}
    \phi_{Q^2}(x) \sim -\frac{1}{|x|} \quad \text{as} \;\; x \to \infty.
\end{equation}
If we consider a two-soliton
\begin{equation}
    u= \frac{1}{\lambda_1^2} Q \left( \frac{x-\alpha_1}{\lambda_1} \right) e^{i\gamma_1+ i\beta_1 \cdot x}+ \frac{1}{\lambda_2^2} Q \left( \frac{x-\alpha_2}{\lambda_2} \right) e^{i\gamma_2+ i\beta_2 \cdot x},
\end{equation}
then the nonlinearity in \eqref{eq hartree} naturally produces a larger interaction error than in NLS.

It was shown in \cite{KMR2bodyHartree} that the effective dynamics of two-solitons of the Hartree equation is governed by the two-body problem
\begin{equation} \label{eq 2-body}
    \dot{\alpha}_j= 2\beta_j, \quad \dot{\beta}_j= -\frac{\| Q \|_{L^2}^2}{4\pi \lambda_{j+1}} \frac{\alpha_j- \alpha_{j+1}}{|\alpha_j- \alpha_{j+1}|^3}, \quad j=1,2,
\end{equation}
where $\alpha_1, \alpha_2, \beta_1, \beta_2 \in C^1(\R,\R^3)$ and $\lambda_1, \lambda_2>0$ are constants. We use the convention that indices are taken modulo $2$, so that $j+1=1$ when $j=2$. Here, the Kepler potential comes from the leading long-range interaction between the two copies of $Q$. This approximation remains valid for $m$-solitons with arbitrary $m \ge 2$ as was recently generalized in \cite{Hartree3Dmultisoliton}; the effective dynamics of $m$-solitons is governed by the $m$-body problem.

In \eqref{eq 2-body}, the center of mass evolves freely: $\lambda_1 \ddot{\alpha}_2 + \lambda_2 \ddot{\alpha}_1= 0$. We will normalize the center of mass by assuming
\begin{equation}
    \lambda_1 \alpha_2+ \lambda_2 \alpha_1=0
\end{equation}
throughout the paper. The two-body problem also admits the conserved energy
\begin{equation}
    E_0:= |\beta_2- \beta_1|^2- \frac{\|Q\|_{L^2}^2}{4\pi} \left( \frac{1}{\lambda_1}+ \frac{1}{\lambda_2} \right) \frac{1}{|\alpha_2- \alpha_1|}.
\end{equation}
It follows that if the two bodies do not collide initially, that is, $\alpha_1(0) \neq \alpha_2(0)$, then they remain away from each other for all time. Let $\alpha= \alpha_2- \alpha_1$. Then global solutions to \eqref{eq 2-body} can be classified according to the sign of $E_0$ or, equivalently, the growth rate of $|\alpha(t)|$:

\begin{itemize}
    \item Hyperbolic: if $E_0>0$, then $|\alpha(t)| \sim t$ as $t \to \infty$ and $\alpha(t)$ describes a hyperbola.
    \item Parabolic: if $E_0=0$, then $|\alpha(t)| \sim t^{\frac{2}{3}}$ as $t \to \infty$ and $\alpha(t)$ describes a parabola.
    \item Elliptic: if $E_0<0$, then $|\alpha(t)|$ is bounded and $\alpha(t)$ describes an ellipse.
\end{itemize}

In both the hyperbolic case and the parabolic case, we have $|\alpha(t)| \to \infty$ as $t \to \infty$, so we say these are \emph{expansive} cases. For expansive dynamics, the separation between the two solitons tends to infinity, weakening the soliton interaction over time. Therefore, one may hope to construct solutions to the Hartree equation \eqref{eq hartree} that, as time goes to infinity, approach two-solitons or multisolitons whose center dynamics is described by an expansive solution to the two-body problem or the $m$-body problem. In particular, \cite{KMR2bodyHartree} constructed two-soliton solutions to the Hartree equation whose center dynamics approaches an expansive solution to \eqref{eq 2-body}, while \cite{Hartree3Dmultisoliton} constructed multisoliton solutions whose trajectories approach some classes of expansive solutions to the $m$-body problem. The recent parallel results for the four-dimensional Hartree equation in \cite{Hartree4Dmultisoliton, Hartree4Dmixed} also considered expansive solutions to the adapted $m$-body problem. In general, expansive dynamics is more approachable because of the decay of the interaction error in time, and known results on multisoliton solutions of Hartree equations have mainly focused on expansive trajectories.

By contrast, in the elliptic case, the solitons remain at bounded separation, so the long-range interaction between the solitons persists. One cannot expect the small interaction errors to decay by waiting for large times. This makes the elliptic trajectory much more delicate to study. Until now, there has been no construction of solutions to the Hartree equation \eqref{eq hartree} whose asymptotic behavior is governed by elliptic solutions to \eqref{eq 2-body}. 

\;\\
\textbf{The main result:}

In this paper, we consider two-soliton solutions to \eqref{eq hartree} with elliptic trajectories. To the author's knowledge, this is the first work to study elliptic center dynamics for the Hartree equation. Even in strongly interacting multisoliton regimes for other models such as \cite{energycriticalwave, Kinkantikinkonaline}, there is usually still some decay of the interaction error in time. Therefore, the elliptic regime considered in this paper represents a genuinely new dynamical setting. 

Our result proves the finite-time stability of two-soliton solutions with approximate elliptic trajectories. As has been explained, the interaction remains of fixed size along the elliptic orbit, so the smallness of the error must come from the order of the approximation rather than from time decay. We thus do not study asymptotic behavior as time tends to infinity; instead, we prove finite-time stability for carefully chosen profiles.

\begin{thm} \label{thm finite time stability}
Given $\delta>0$, $C_e>0$ and $\lambda_1, \lambda_2>0$, there exist $\nu>0$, $\bar{r}_0>0$ such that for any $r_0> \bar{r}_0$, there exists $\kappa>0$ with the following property. 

If $(\alpha_1(t), \alpha_2(t), \beta_1(t), \beta_2(t), \lambda_1, \lambda_2)$ is an elliptic solution to \eqref{eq 2-body} with 
\begin{equation}
    r_0 \le |\alpha_2(t)- \alpha_1(t)| \le C_e r_0, \quad \forall t \ge 0
\end{equation} 
and $\gamma_1(t)$, $\gamma_2(t)$ satisfy
\begin{equation}
    \dot{\gamma}_j= \frac{1}{\lambda_j^2}- |\beta_j|^2- \dot{\beta}_j \cdot \alpha_j+ \frac{\| Q \|_{L^2}^2}{4\pi \lambda_{j+1}} \frac{1}{|\alpha_2- \alpha_1|}, \quad j=1,2,
\end{equation} 
then there exists $R_0 \in H^1(\R^3)$ with 
\begin{equation} \label{eq R_0 close to initial}
    \left\| R_0 - \sum_{j=1}^2 \frac{1}{\lambda_j^2} Q \Big( \frac{\cdot- \alpha_j(0)}{\lambda_j} \Big) e^{i\gamma_j(0)+ i\beta_j(0) \cdot x} \right\|_{H^1} \lesssim_{\lambda_1, \lambda_2} r_0^{-3}
\end{equation}
such that for $u_0 \in H^1(\R^3)$, $\| u_0- R_0 \|_{H^1} < \kappa$ implies
\begin{equation} \label{eq estimate in thm}
    \left\| u(t) - \sum_{j=1}^2 \frac{1}{\lambda_j^2} Q \Big( \frac{\cdot- \alpha_j(t)}{\lambda_j} \Big) e^{i\gamma_j(t)+ i\beta_j(t) \cdot x} \right\|_{H^1} < \delta, \quad \forall 0 \le t \le \nu r_0^2,
\end{equation}
where $u$ is the solution to \eqref{eq hartree} with $u(0)= u_0$.
\end{thm}

\begin{rmk}
The period of the ellipse $\alpha_2(t)- \alpha_1(t)$ has size $r_0^{3/2}$, so the time scale $r_0^2$ on which we establish stability is much longer than the period. In other words, the two-soliton solution will remain close to the elliptic orbit for many cycles before, if ever, departing from it.
\end{rmk}

\begin{rmk}
The $\frac{\| Q \|_{L^2}^2}{4\pi \lambda_{j+1}} \frac{1}{|\alpha_2- \alpha_1|}$ term in the equation for $\gamma_j$ is nonstandard. We call it an \emph{interaction shift} and it is necessary in the construction of approximate solutions. See Remark \ref{rmk interaction shift} for further explanation.
\end{rmk}

\begin{rmk}
The assumption $r_0 \le |\alpha_2(t)- \alpha_1(t)| \le C_e r_0$ means that the elliptic trajectory has size $r_0$ and uniformly bounded eccentricity. The eccentricity bound is necessary for establishing the stability of the center dynamics in Proposition \ref{prop stability of 2-body}.
\end{rmk}

\begin{rmk}
We will construct approximate profiles $V_1, V_2 \in H^1(\R^3)$ close to $Q$ and take 
\begin{equation}
    R_0(x)= \sum_{j=1}^2 \frac{1}{\lambda_j^2} V_j \Big( \frac{x- \alpha_j(0)}{\lambda_j} \Big) e^{i\gamma_j(0)+ i\beta_j(0) \cdot x}.
\end{equation}
Here $V_1$ and $V_2$ encode higher-order correction terms compensating for the long-range effect of the Hartree equation. This is why $R_0$ is only required to be within $O(r_0^{-3})$ in $H^1$ of the naive two-soliton profile in \eqref{eq R_0 close to initial}.

As shown in the proof of the theorem in Section \ref{sec error estimates}, the profiles $V_1, V_2$ depend on 
\begin{equation}
    \alpha_2(0)- \alpha_1(0), \;\; \beta_2(0)- \beta_1(0), \;\; \lambda_1, \;\; \lambda_2, \;\; \gamma_1(0), \;\; \gamma_2(0).
\end{equation}
Consequently, $R_0$ depends on the initial parameters 
\begin{equation}
    \alpha_1(0), \;\; \alpha_2(0), \;\; \beta_1(0), \;\; \beta_2(0), \;\; \lambda_1, \;\; \lambda_2, \;\; \gamma_1(0), \;\; \gamma_2(0).
\end{equation}
\end{rmk}

The proof follows the same broad modulation strategy as \cite{KMR2bodyHartree,Hartree3Dmultisoliton}, but the elliptic setting changes both the role of the approximate solution and the order to which it must be constructed. In the expansive cases considered there, the separation between solitons tends to infinity, so interaction errors tend to $0$ as time tends to $\infty$. Therefore, when the order of approximation is high enough, the error decays nearly exponentially in time. In the elliptic case in this paper, the separation remains comparable to the fixed parameter $r_0$ along the elliptic orbit. Thus the main point is to construct an approximate solution whose error is small uniformly on the relevant time interval. By introducing a cutoff in \eqref{eq definition of psi^n,L} when defining approximate solutions and keeping track of the dependence of all constants on the approximation order $N$, it turns out that $N$ should be chosen proportional to the separation $r_0$ (see Remark \ref{rmk range of N, L}), although this relation is not apparent from the approximation scheme itself. 

More precisely, the proof has three main ingredients. First, in Section \ref{sec approximate solution}, we construct an $N$-th order approximate two-soliton and prove that its error is exponentially small in $r_0$ when $N \sim r_0$. Second, in Section \ref{sec ode analysis}, we show that the approximate center dynamics remains close to the elliptic Kepler dynamics on the required time scale. Third, in Section \ref{sec error estimates}, we use modulation and a localized coercive energy functional to control the difference between the true solution and the approximate profile. We also collect preparatory properties of the linearized operators in Section \ref{sec linear operator}.

The construction scheme of approximate solutions and the localized energy estimate are closest to \cite{KMR2bodyHartree,Hartree3Dmultisoliton}. The new difficulties arise from the bounded separation of elliptic trajectories. They are addressed by the lower-order cancellations in Proposition \ref{prop no lower order error} using the \emph{interaction shift} in Remark \ref{rmk interaction shift}, the quantitative accuracy estimate in Proposition \ref{prop accuracy of approximate solution} using the cutoff in the dipole expansion \eqref{eq definition of psi^n,L}, and the stability analysis of the approximate elliptic center dynamics in Proposition \ref{prop stability of 2-body}.

\section{Linearized operators around the ground state} \label{sec linear operator}

In this section, we study some properties of the linearized operators around the ground state. Write $u=e^{it} (Q+\varepsilon)$. Using \eqref{eq ground state}, we can rewrite \eqref{eq hartree} as
\begin{equation} 
    \partial_t \begin{pmatrix}
        \Re \ \varepsilon \\
        \Im \ \varepsilon
    \end{pmatrix}+ \begin{pmatrix}
        0 &  -L_- \\
        L_+ & 0
    \end{pmatrix} \begin{pmatrix}
        \Re \ \varepsilon \\
        \Im \ \varepsilon
    \end{pmatrix}= \begin{pmatrix}
        \Im \ \mathcal{N}(Q, \varepsilon) \\
        -\Re \ \mathcal{N}(Q, \varepsilon)
    \end{pmatrix},
\end{equation}
where $\mathcal{N}$ is of second or higher order in $\varepsilon$, and $L_+, L_-$, the linearized operators around $Q$, are defined by
\begin{equation}
    L_+f:= -\Delta f+ f+ \phi_{Q^2}f+ 2\phi_{Qf}Q, \quad L_-f:= -\Delta f+ f+ \phi_{Q^2}f.
\end{equation}

The kernels of $L_{\pm}$ necessarily contain modes generated by the symmetries of \eqref{eq hartree}. From \cite{Lenzmanngroundstate}, we know the ground state $Q$ is non-degenerate, meaning that the kernels of $L_{\pm}$ consist precisely of those symmetry directions:
\begin{equation}
    \ker (L_+)= \text{span} \{\partial_1 Q, \partial_2 Q, \partial_3 Q\}, \quad \ker(L_-)= \text{span} \{ Q \}.
\end{equation}

Let $\Lambda= 2+ x \cdot \nabla$. We recall \cite[Lemma 4.4]{Hartree3Dmultisoliton}.
\begin{lem} \label{lem coercivity of L+, L-}
There exist $\delta,c>0$ such that if $v \in H^1$ is real-valued, then
\begin{equation} \begin{gathered} 
    |(v,Q)|+ |(v,xQ)|< \delta \Vert v \Vert_{H^1} \implies (L_+v,v) \ge c \Vert v \Vert_{H^1}^2, \\
    |(v, \Lambda Q)|< \delta \Vert v \Vert_{H^1} \implies (L_-v,v) \ge c \Vert v \Vert_{H^1}^2.
\end{gathered} \end{equation}
\end{lem}

The main purpose of this section is to prove the following quantitative version of Lemma 2.4 in \cite{KMR2bodyHartree}. In particular, it shows that when $L_\pm^{-1}$ is applied to an exponentially decaying function, the exponential decay rate does not weaken, and the weighted $L^\infty$ norm is uniform.

\begin{lem} \label{lem inverse of L+, L-}
There exist universal constants $c,C>0$ such that the following holds. If $f$ is a real-valued function with $|f(x)| \le e^{-c|x|}$, then
\begin{enumerate} [label=(\arabic*)]
    \item $(f, \nabla Q)=0 \implies \exists u$ such that $L_+u=f$ and $|u(x)| \le Ce^{-c|x|}$;
    \item $(f,Q)=0 \implies \exists u$ such that $L_-u=f$ and $|u(x)| \le Ce^{-c|x|}$.
\end{enumerate}
\end{lem}

\begin{proof}
\cite[Lemma~2.4]{KMR2bodyHartree} asserts the existence of $u$. It remains to show the estimate. Let $C_0>0$ and $c_0>0$ be such that $Q(x) \le C_0 e^{-c_0|x|}$ for all $x \in \R^3$. These are universal constants. We claim that taking $c<\min\{c_0,1\}$ suffices. 

We first prove (1). Since $\nabla Q \in \ker(L_+)$, we have $L_+(u+ a\cdot \nabla Q)=f$ for any $a \in \R^3$. Note $(\nabla Q, xQ)$ is invertible. By choosing $a$ properly so that $(u+a\cdot \nabla Q, xQ)=0$, we may replace $u$ by this representative and assume $(u,xQ)=0$. Then by Lemma \ref{lem coercivity of L+, L-}, one has
\begin{equation}
    (f,u)= (L_+u, u) \ge C^{-1} \| u \|_{H^1}^2- C(u,Q)^2.
\end{equation}
Note that $(u,2Q)= -(u, L_+ (\Lambda Q))= -(f, \Lambda Q)$. We deduce 
\begin{equation}
    \|u\|_{L^2}^2 \le \| u \|_{H^1}^2 \le C(f,u)+ C(f,\Lambda Q)^2.
\end{equation}
By Cauchy--Schwarz, we get the uniform bound
\begin{equation}
    \|u\|_{L^2} \le C.
\end{equation}

The equation for $u$ can be written as
\begin{align}
    u &= (-\Delta+1)^{-1}f- (-\Delta+1)^{-1} \big( \phi_{Q^2}u+ 2\phi_{Qu} Q \big).
\end{align}
In three dimensions, the kernel of $(-\Delta+1)^{-1}$ is $\frac{e^{-|x-y|}}{4\pi |x-y|}$. Then we have
\begin{equation}
    u(x) = \int_{\R^3} \frac{e^{-|x-y|}}{4\pi |x-y|} f(y) \d y- \int_{\R^3} \frac{e^{-|x-y|}}{4\pi |x-y|} \big( \phi_{Q^2}(y) u(y)+ 2\phi_{Qu}(y) Q(y) \big) \d y.
\end{equation}
Fix any $M>0$. We want to estimate $\min \{ M, e^{c|x|} \} |u(x)|$. We first have
\begin{equation}
    e^{c|x|} \int_{\R^3} \frac{e^{-|x-y|}}{4\pi |x-y|} f(y) \d y \le \int_{\R^3} \frac{e^{-(1-c)|x-y|}}{4\pi |x-y|} \d y \le C.
\end{equation}
Take a universal constant $L>0$ such that 
\begin{equation}
    \sup_{|y| \ge L} |\phi_{Q^2}(y)| \int_{\R^3} \frac{e^{-|x|}}{4\pi |x|} \d x< \frac{1}{2}.
\end{equation}
Then using H\"older's inequality and Hardy--Littlewood--Sobolev's inequality, we have
\begin{align}
    &\ e^{c|x|} \int_{|y| \le L} \frac{e^{-|x-y|}}{4\pi |x-y|} |\phi_{Q^2}(y) u(y)| \d y + e^{c|x|} \int_{\R^3} \frac{e^{-|x-y|}}{4\pi |x-y|} |\phi_{Qu}(y) Q(y)| \d y \\
    \le &\ \int_{|y| \le L} \frac{e^{|y|} |\phi_{Q^2}(y)|}{4\pi |x-y|} |u(y)| \d y+ \int_{\R^3} \frac{e^{-(1-c)|x-y|}}{4\pi |x-y|} |\phi_{Qu}(y)| \d y \le C \|u\|_{L^2} \le C.
\end{align}
Combining these estimates, we obtain
\begin{equation}
    \min \{ M, e^{c|x|} \} |u(x)| \le C+ \frac{1}{2} \sup_{y \in \R^3} \min \{ M, e^{c|y|} \} |u(y)|.
\end{equation}
Taking supremum in $x$, we get
\begin{equation}
    \min \{ M, e^{c|x|} \} |u(x)| \le C, \quad \forall x \in \R^3
\end{equation}
with $C>0$ independent of $M$. Then letting $M \to +\infty$ gives the desired estimate.

For (2), using $Q \in \ker(L_-)$, we may choose $u$ so that $L_-u=f$ and $(u,\Lambda Q)=0$. Here, we have used the fact that $(Q, \Lambda Q) \neq 0$. Then Lemma \ref{lem coercivity of L+, L-} implies $(L_-u,u) \ge C^{-1} \|u\|_{H^1}^2$. We then deduce $\|u\|_{L^2} \le C$ similarly. The remaining proof is identical to (1).
\end{proof}

\section{Approximate two-soliton solutions} \label{sec approximate solution}

In this section, we construct approximate two-soliton solutions to the Hartree equation \eqref{eq hartree} in Proposition \ref{prop construction of approximate bubbles} and prove quantified accuracy estimates for approximate solutions in Proposition \ref{prop accuracy of approximate solution}. The strategy was introduced in \cite{KMR2bodyHartree} and was applied to multisolitons in \cite{Hartree3Dmultisoliton}; our notation is closer to that of \cite{Hartree3Dmultisoliton}. However, we make a few key modifications. First, we introduce a cutoff function in \eqref{eq definition of psi^n,L} when defining the dipole expansion. Second, we add an additional term in \eqref{eq definition of E_j tilde} and \eqref{eq definition of S_j^N} when splitting the interaction error. Third, we introduce a quantitative version of admissible functions in Definition \ref{def adm functions quant}. These modifications lead to Proposition \ref{prop no lower order error} and Proposition \ref{prop accuracy of approximate solution}, which are new ingredients of this paper as mentioned in Section \ref{sec intro}.

Let us introduce some notation. When given parameters $\alpha_1, \alpha_2, \beta_1, \beta_2, \lambda_1, \lambda_2, \gamma_1, \gamma_2$, we write
\begin{gather}
    P=(\alpha_1, \alpha_2, \beta_1, \beta_2, \lambda_1, \lambda_2), \quad g=(P, \gamma_1, \gamma_2), \\
    g_1=(\alpha_1, \beta_1, \lambda_1, \gamma_1), \quad g_2=(\alpha_2, \beta_2, \lambda_2, \gamma_2), \\
    \alpha= \alpha_2- \alpha_1, \quad \beta= \beta_2- \beta_1, \quad \lambda=(\lambda_1, \lambda_2), \quad r=|\alpha|.
\end{gather}
We use similar notation for superscripted quantities. For a function $v$, let
\begin{equation}
    g_j v(t,x)= \frac{1}{\lambda_j^2} v \Big( t, \frac{x-\alpha_j}{\lambda_j} \Big) e^{i\gamma_j+ i\beta_j \cdot x}, \quad j=1,2.
\end{equation}

Consider the $2$-soliton
\begin{equation}
    u(t,x):= u_1+ u_2:= g_1 v_1+ g_2 v_2.
\end{equation}
Then we have
\begin{equation}
    iu_t+ \Delta u -\phi_{|u|^2} u = \sum_{j=1}^2 \frac{1}{\lambda_j^2} E_j(t,y_j) e^{i\gamma_j+ i\beta_j \cdot x}- 2\phi_{\Re (u_1 \ol{u_2})}u
\end{equation}
with
\begin{equation} \begin{aligned}
    E_j(t,y_j)= &\quad i\partial_t v_j+ \lambda_j^{-2} \big( \Delta v_j- v_j- \phi_{|v_j|^2} v_j \big)- \frac{1}{\lambda_{j+1}^2} \phi_{|v_{j+1}|^2} \left( t, \frac{\lambda_j y_j}{\lambda_{j+1}}+ \frac{(-1)^j\alpha}{\lambda_{j+1}} \right) v_j \\
    &- i \frac{\dot{\alpha}_j- 2\beta_j}{\lambda_j} \nabla v_j- i \frac{\dot{\lambda}_j}{\lambda_j} \Lambda v_j- \lambda_j \dot{\beta}_j \cdot y_j v_j- \Big( \dot{\gamma}_j- \frac{1}{\lambda_j^2}+ |\beta_j|^2+ \dot{\beta}_j \cdot \alpha_j \Big) v_j,
\end{aligned} \end{equation}
where we set $y_j= \frac{x-\alpha_j}{\lambda_j}$. Note that
\begin{equation}
    \frac{1}{\lambda_{j+1}^2} \phi_{|v_{j+1}|^2} \left( t, \frac{\lambda_j y_j}{\lambda_{j+1}}+ \frac{(-1)^j \alpha}{\lambda_{j+1}} \right) = -\frac{1}{4\pi \lambda_{j+1}} \int_{\mathbb{R}^3} \frac{|v_{j+1}(t,\xi)|^2}{|\lambda_j y_j+ (-1)^j \alpha- \lambda_{j+1} \xi|} \d \xi.
\end{equation}
To construct approximate solutions, we consider the Taylor expansion 
\begin{equation}
    \frac{1}{|\alpha- \zeta|}= \sum_{n=1}^N F_n(\alpha,\zeta)+ O \left( \frac{|\zeta|^N}{|\alpha|^{N+1}} \right) \;\; \text{as} \;\; \zeta \to 0,
\end{equation}
where $F_n(\alpha,\zeta)$ is homogeneous of degree $-n$ in $\alpha$ and of degree $n-1$ in $\zeta$. We can calculate the first few terms explicitly:
\begin{equation}
    F_1(\alpha, \zeta)= \frac{1}{|\alpha|}, \quad F_2(\alpha, \zeta)= \frac{\alpha \cdot \zeta}{|\alpha|^3}, \quad F_3(\alpha, \zeta)= \frac{1}{2|\alpha|^5} \big[ 3(\alpha \cdot \zeta)^2- |\alpha|^2 |\zeta|^2 \big].
\end{equation}
As in \cite{Hartree3Dmultisoliton}, we define
\begin{equation}
    \psi_{|v_{j+1}|^2}^{(n)} (t,y_j)= -\frac{1}{4\pi \lambda_{j+1}} \int_{\R^3} |v_{j+1}(t,\xi)|^2 F_n \big( \alpha, (-1)^j (\lambda_{j+1} \xi- \lambda_j y_j) \big) \d \xi.
\end{equation}
When $v_{j+1}=Q$, we denote $\psi_{|v_{j+1}|^2}^{(n)}$ by $\psi_{Q^2,j+1}^{(n)}$ to indicate its dependence on $j$.

Note that the Taylor expansion is not useful when $\zeta=(-1)^j (\lambda_{j+1}\xi-\lambda_jy_j)$ is large compared to $|\alpha|$. In \cite{KMR2bodyHartree,Hartree3Dmultisoliton}, using the expansion for all $\zeta$ is harmless because one can first take $N$ large enough and then take $|\alpha|$ large enough, as $|\alpha(t)| \to \infty$ in the expansive case, thus absorbing the polynomial growth by the exponential decay. However, this is impossible in the elliptic case where $|\alpha|$ is bounded for a fixed trajectory. Our solution is to introduce a cutoff so that the approximation only captures small $\zeta$. In practice, it is more convenient to let the cutoff function be a function of $y_j$. Another subtle point is that we will need precise formulae for the first few terms of the approximation as in Proposition \ref{prop no lower order error}, so we do not want the cutoff to affect the first few terms. Therefore, we only apply the cutoff for $n>10$, where the choice of $10$ is quite flexible: it can be any order above which explicit expressions for the approximation are not required. 

More precisely, we fix a smooth function $\chi:\R^3 \to \R_{\ge 0}$ such that $\chi(x)=1$ if $|x| \le 1$ and $\chi(x)=0$ if $|x| \ge 2$. Then we define
\begin{equation} \label{eq definition of psi^n,L}
    \psi_{|v_{j+1}|^2}^{(n),L} (t,y_j)= \left\{
    \begin{aligned}
        &\psi_{|v_{j+1}|^2}^{(n)} (t,y_j), && n \le 10, \\
        &\psi_{|v_{j+1}|^2}^{(n)} (t,y_j) \chi(\frac{y_j}{L}), && n>10,
    \end{aligned} \right.
\end{equation}
and
\begin{equation} 
    \phi_{|v_{j+1}|^2}^{(N),L}(t,y_j) := \sum_{n=1}^N \psi_{|v_{j+1}|^2}^{(n),L}(t,y_j),
\end{equation}
where $L>0$ is a constant that will be specified later.

We shall allow $v_j$ to depend on $N$, and we also assume $v_j$ depends on time only through the parameters $t \mapsto P(t)$, which means $v_j(t,y_j)= V_j^{(N)}(P(t),y_j)$ for some $V_j^{(N)}$. Define
\begin{equation} \label{eq approximate solution} \begin{aligned}
    R_g^{(N)}(t,x):= \sum_{j=1}^2 R_{j,g}^{(N)}(t,x):= \sum_{j=1}^2 g_j V_j^{(N)}(P(t),x). 
\end{aligned} \end{equation}
Let us omit the subscript $g$ of $R^{(N)}$ for now. We have
\begin{equation} \label{eq equation of approximate solution} \begin{aligned}
    &i\partial_t R^{(N)}+ \Delta R^{(N)}- \phi_{|R^{(N)}|^2} R^{(N)} \\
    = &\quad \sum_{j=1}^2 \frac{1}{\lambda_j^2} E_j^{(N)}(t,y_j) e^{i\gamma_j+ i\beta_j \cdot x}- 2 \phi_{\Re \Big( R_1^{(N)} \ol{R_2^{(N)}} \Big)} R^{(N)} \\
    &+ \sum_{j=1}^2 \frac{1}{\lambda_j^2} \left[ \phi_{\left| V_{j+1}^{(N)} \right|^2}^{(N),L}- \frac{1}{\lambda_{j+1}^2} \phi_{|V_{j+1}^{(N)}|^2} \left( P(t), \frac{\lambda_j y_j}{\lambda_{j+1}}+ \frac{(-1)^j \alpha}{\lambda_{j+1}} \right) \right] V_j^{(N)} e^{i\gamma_j+ i\beta_j \cdot x},
\end{aligned} \end{equation} 
where $E_j^{(N)}= \tilde{E}_j^{(N)}+ S_j^{(N)}$, with $\tilde{E}_j^{(N)}$ collecting nonlinear errors and $S_j^{(N)}$ collecting modulation errors. Namely, we let
\begin{equation} \label{eq definition of E_j tilde} \begin{aligned} 
    \tilde{E}_j^{(N)}(t,y_j)= & \lambda_j^{-2} \Big( \Delta V_j^{(N)}- V_j^{(N)}- \phi_{\left| V_j^{(N)} \right|^2} V_j^{(N)} \Big)- \phi_{\left| V_{j+1}^{(N)} \right|^2}^{(N),L} V_j^{(N)} \\
    &- i \frac{M_j^{(N)}}{\lambda_j} \Lambda V_j^{(N)}- \lambda_j B_j^{(N)} \cdot y_j V_j^{(N)}- \frac{\| Q \|_{L^2}^2}{4\pi \lambda_{j+1}} \frac{1}{|\alpha|} V_j^{(N)} \\
    &+ i\Bigg[ \frac{\partial V_j^{(N)}}{\partial \alpha} \cdot 2\beta+ \frac{\partial V_j^{(N)}}{\partial \beta} \cdot \left( B_2^{(N)}- B_1^{(N)} \right)+ \sum_{k=1}^2 \frac{\partial V_j^{(N)}}{\partial \lambda_k}  M_k^{(N)} \Bigg]
\end{aligned} \end{equation} 
and
\begin{equation} \label{eq definition of S_j^N} \begin{aligned}
    S_j^{(N)}(t, y_j) = &- i \frac{\dot{\alpha}_j- 2\beta_j}{\lambda_j} \nabla V_j^{(N)}- \lambda_j \left( \dot{\beta}_j- B_j^{(N)} \right) \cdot y_j V_j^{(N)}- i\frac{\dot{\lambda}_j - M_j^{(N)}}{\lambda_j} \Lambda V_j^{(N)} \\
    &- \Big( \dot{\gamma}_j- \frac{1}{\lambda_j^2}+ |\beta_j|^2+ \dot{\beta}_j \cdot \alpha_j-\frac{\| Q \|_{L^2}^2}{4\pi \lambda_{j+1}} \frac{1}{|\alpha|} \Big) V_j^{(N)} \\
    &+ i \Bigg[ \frac{\partial V_j^{(N)}}{\partial \alpha} \cdot \left( \dot{\alpha}- 2\beta \right)+ \frac{\partial V_j^{(N)}}{\partial \beta} \cdot \left( \dot{\beta}- B_2^{(N)}+ B_1^{(N)} \right) \Bigg] \\
    &+ i \sum_{k=1}^2 \frac{\partial V_j^{(N)}}{\partial \lambda_k} \left( \dot{\lambda}_k- M_k^{(N)} \right).
\end{aligned} \end{equation}

\begin{rmk} \label{rmk interaction shift}
This error expansion is almost the same as what was done in \cite{KMR2bodyHartree, Hartree3Dmultisoliton}, except that we add an additional term $-\frac{\| Q \|_{L^2}^2}{4\pi \lambda_{j+1}} \frac{1}{|\alpha|} V_j^{(N)}$ in \eqref{eq definition of E_j tilde} and a corresponding term in \eqref{eq definition of S_j^N}. We will call this term an \textbf{interaction shift} since it modifies the equation for $\gamma_j$ when we later impose the condition $S_j^{(N)}=0$. This is reflected in both Theorem \ref{thm finite time stability} and Proposition \ref{prop finite time stability with N}. 

The interaction shift cancels the first-order interaction term, and, with the aid of spherical harmonics, leads to further higher-order cancellations. The first-order calculation in the proof of Proposition \ref{prop no lower order error}, given in Section \ref{sec ode analysis}, contains the details of this cancellation.

The interaction shift also appears in the Lagrangian formulation of the effective dynamics. In this derivation, the same quantity enters the equation for $\gamma_j$, providing an alternative interpretation of the term introduced above. Since the calculation is lengthy and plays no role in the proof, we omit the details.
\end{rmk}

We recall the following definition of admissible functions from \cite{KMR2bodyHartree, Hartree3Dmultisoliton}. These functions have good decay in both $|\alpha|$ and $x$. 

\begin{defn}[\textbf{Admissible functions}] Let $n \in \mathbb{N}_+$.

(1) Define $S_n$ to be the set of functions $\sigma: (\R^3 \setminus \{0\}) \times \R^3 \times \R_+ \times \R_+ \to \mathbb{C}$ of the form
\begin{equation}
    \sigma(\alpha, \beta, \lambda_1, \lambda_2)= \sum z_{s,t,p,q} |\alpha|^{-n-|s|} \alpha^s \beta^t \lambda_1^p \lambda_2^q,
\end{equation}
where the sum is finite, $s, t \in \N^3$, $p,q \in \mathbb{Z}$ and $z_{s,t,p,q} \in \C$.

(2) We say a function $u: (\R^3 \setminus \{0\}) \times \R^3 \times \R_+ \times \R_+ \times \R^3 \to \mathbb{C}$ is admissible of degree $n$ if it is of the form
\begin{equation}
    u(\alpha, \beta, \lambda_1, \lambda_2,x)= \sum z_{s,t,p,q} |\alpha|^{-n-|s|} \alpha^s \beta^t \lambda_1^p \lambda_2^q f_{s,t,p,q}(x),
\end{equation}
where the sum is finite, $s, t \in \N^3$, $p,q \in \mathbb{Z}$ and
\begin{equation} \label{eq estimate for admissible functions unquantified}
    z_{s,t,p,q} \in \C; \quad |\nabla^k f_{s,t,p,q}(x)| \le C_{k,s,t,p,q} e^{-c_{k,s,t,p,q}|x|}, \;\; \forall k \in \N, \; x \in \R^3.
\end{equation}

(3) We say $u: (\R^3 \setminus \{0\}) \times \R^3 \times \R_+ \times \R_+ \times \R^3 \to \mathbb{C}$ is admissible of degree $\ge n$ if $u$ is a finite sum of functions of degree $n'$ with $n' \ge n$.
\end{defn}

Approximate solutions can then be constructed as follows.
\begin{prop} \label{prop construction of approximate bubbles}
There exist $b_j^{(n)}, m_j^{(n)} \in S_n$ that are real-valued and $T_j^{(n)}$ that is admissible of degree $n$ for $n \ge 1$, $j=1,2$ such that: for any $N \ge 1$, if we let 
\begin{gather*}
    V_j^{(N)}(\alpha, \beta, \lambda_1, \lambda_2, y_j)= Q(y_j)+ \sum_{n=1}^N T_j^{(n)}(\alpha, \beta, \lambda_1, \lambda_2,y_j), \\
    M_j^{(N)}= \sum_{n=1}^N m_j^{(n)} \quad \text{and} \quad B_j^{(N)}= \sum_{n=1}^N b_j^{(n)},    
\end{gather*}
then $\tilde{E}_j^{(N)}$ defined by \eqref{eq definition of E_j tilde} is admissible of degree $\ge N+1$.
\end{prop}

\begin{proof}
The proof is identical to \cite[Proposition~2.3]{KMR2bodyHartree} or \cite[Proposition~2.4]{Hartree3Dmultisoliton}, despite the additional term $- \frac{\| Q \|_{L^2}^2}{4\pi \lambda_{j+1}} \frac{1}{|\alpha|} V_j^{(N)}$ and the cutoff \eqref{eq definition of psi^n,L}. The construction can be carried out in the same way. We will only provide a sketch and explain how to determine $m_j^{(n)}, b_j^{(n)}$ and $T_j^{(n)}$.

Let $\hat{E}_j^{(N)}$ denote the terms in $\tilde{E}_j^{(N)}$ that are of degree $-N-1$ in $\alpha$. We have
\begin{equation} \begin{aligned} \label{eq E_j^N hat}
    \hat{E}_j^{(N)}= &- \lambda_j^{-2} \sum_{k+l+m=N+1} \phi_{\Re \big( T_j^{(k)} \ol{T_j^{(l)}} \big)} T_j^{(m)}- \sum_{k+l+m+n=N+1} \psi^{(n),L}_{\Re \big( T_{j+1}^{(k)} \ol{T_{j+1}^{(l)}} \big)} T_j^{(m)} \\
    &- \frac{i}{\lambda_j} \sum_{m+n=N+1} m_j^{(n)} \Lambda T_j^{(m)} - \lambda_j \sum_{m+n=N+1} b_j^{(n)} \cdot y_j T_j^{(m)}- \frac{\| Q \|_{L^2}^2}{4\pi \lambda_{j+1}} \frac{1}{|\alpha|} T_j^{(N)} \\
    &+i \frac{\partial T_j^{(N)}}{\partial \alpha} \cdot 2\beta+ i\sum_{m+n=N+1} \bigg[ \frac{\partial T_j^{(m)}}{\partial \beta} \cdot \Big( b_2^{(n)}- b_1^{(n)} \Big)+ \frac{\partial T_j^{(m)}}{\partial \lambda_1} m_1^{(n)}+ \frac{\partial T_j^{(m)}}{\partial \lambda_2} m_2^{(n)} \bigg],
\end{aligned} \end{equation}
where the indices $k,l,m$ range from $0$ to $N$, $n$ ranges from $1$ to $N$, and we set $T_j^{(0)}=Q$. 

Then we choose the $(N+1)$-th order terms to cancel $\hat{E}_j^{(N)}$. We let $b_j^{(N+1)}$, $m_j^{(N+1)}$ satisfy 
\begin{equation} \label{eq formula of m_j b_j}
    \left\{ \begin{aligned}
        &\Big( \lambda_j b_j^{(N+1)} \cdot y_j Q+ \psi_{Q^2,j+1}^{(N+1),L} Q- \Re \ \hat{E}_j^{(N)}, \nabla Q \Big)=0, \\
        &\left( m_j^{(N+1)} \Lambda Q- \lambda_j \Im \ \hat{E}_j^{(N)}, Q \right)=0.
    \end{aligned} \right.
\end{equation}
and then solve $T_j^{(N+1)}$ from
\begin{equation} \label{eq formula of T_j}
    \left\{ \begin{aligned}
        &L_+ \Re \ T_j^{(N+1)}= -\lambda_j^3 b_j^{(N+1)} \cdot y_j Q- \lambda_j^2 \psi_{Q^2,j+1}^{(N+1),L} Q+ \lambda_j^2 \Re \ \hat{E}_j^{(N)}, \\
        &L_- \Im \ T_j^{(N+1)}= -\lambda_j m_j^{(N+1)} \Lambda Q+ \lambda_j^2 \Im \ \hat{E}_j^{(N)}.
    \end{aligned} \right.
\end{equation}
Here, the orthogonality condition \eqref{eq formula of m_j b_j} ensures that \eqref{eq formula of T_j} is solvable according to Lemma \ref{lem inverse of L+, L-}. Then the proposition can be proved by induction. See \cite{Hartree3Dmultisoliton} for more details.
\end{proof}

Using \eqref{eq E_j^N hat}, \eqref{eq formula of m_j b_j} and \eqref{eq formula of T_j}, we can compute the first few terms of $B_j^{(N)}$, $M_j^{(N)}$ and $T_j^{(N)}$. These terms turn out to be unexpectedly simple.
\begin{prop} \label{prop no lower order error}
We have $T_j^{(1)}= T_j^{(2)}=0$, $b_j^{(2)}= \frac{(-1)^{j-1}\| Q \|_{L^2}^2}{4\pi \lambda_{j+1}} \frac{\alpha}{|\alpha|^3}$ and 
\begin{equation}
    b_j^{(n)}=0 \;\; \text{for} \;\; 1 \le n \le 6, \; n \neq 2; \quad m_j^{(n)}=0 \;\; \text{for} \;\; 1 \le n \le 6.
\end{equation}
\end{prop}

The significance of this result will become clear in Section \ref{sec ode analysis} when we prove Proposition \ref{prop stability of 2-body}, so we postpone its proof to the next section. We would like to emphasize that the key reason this holds is the interaction shift (see Remark \ref{rmk interaction shift}).

The next step is to prove the accuracy of approximate solutions similar to \cite[Proposition~2.7]{KMR2bodyHartree} or \cite[Proposition~2.6]{Hartree3Dmultisoliton}. However, in those papers, the dependence of the constants on $N$ is not tracked, while this information is needed in this paper. For this purpose, we need to define a more quantitative version of admissible functions and show quantitative estimates for the data constructed above.

\begin{defn}[\textbf{Quantified admissible functions}] \label{def adm functions quant} Let $n \in \mathbb{N}_+$, $A>0$ and $c>0$.

(1) Define $S_n(A)$ to be the set of functions 
\begin{equation}
    \sigma: (\R^3 \setminus \{0\}) \times \R^3 \times \R_+ \times \R_+ \to \mathbb{C}
\end{equation} 
of the form
\begin{equation}
    \sigma(\alpha, \beta, \lambda_1, \lambda_2)= \sum z_{s,t,p,q} |\alpha|^{-n-|s|} \alpha^s \beta^t \lambda_1^p \lambda_2^q,
\end{equation}
where the summation is over 
\begin{equation} \label{eq summation over}
    s \in \N^3, \; |s| \le n; \quad t \in \N^3, \; |t| \le n; \quad p \in \mathbb{Z}, \; |p| \le 5n; \quad q \in \mathbb{Z}, \; |q| \le 5n,
\end{equation} 
and
\begin{equation}
    z_{s,t,p,q} \in \C, \quad
    \sum |z_{s,t,p,q}| \le A.
\end{equation}

(2) Define $\mathcal{S}_n(A,c)$ to be the set of functions 
\begin{equation}
    u: (\R^3 \setminus \{0\}) \times \R^3 \times \R_+ \times \R_+ \times \R^3 \to \mathbb{C}
\end{equation} 
of the form
\begin{equation}
    u(\alpha, \beta, \lambda_1, \lambda_2,x)= \sum z_{s,t,p,q} |\alpha|^{-n-|s|} \alpha^s \beta^t \lambda_1^p \lambda_2^q f_{s,t,p,q}(x),
\end{equation}
where the summation is over \eqref{eq summation over}, with \eqref{eq estimate for admissible functions unquantified} being satisfied, and 
\begin{equation}
    \sum |z_{s,t,p,q}| \le A, \quad  |f_{s,t,p,q}(x)| \le e^{-c|x|}, \;\; \forall x \in \R^3.
\end{equation}
\end{defn}

The following lemma gives quantitative bounds for the closure properties of admissible functions. We remark that (3) below heavily relies on \eqref{eq definition of psi^n,L}.

\begin{lem} \label{lem admissible functions}
Let $m,n \in \N_+$ and $c_1, c_2>0$. Assume $u \in \mathcal{S}_m(A,c_1)$ and $v \in \mathcal{S}_n(B,c_2)$. Then

\begin{enumerate}
    \item $uv \in \mathcal{S}_{m+n}(AB,c_1+c_2)$;
    \item $\phi_u v \in \mathcal{S}_{m+n}(CAB,c_2)$ for some $C>0$ depending only on $c_1$;
    \item If $0< \epsilon \ll 1$, $k \in \N_+$ and $k \le L$, then $\psi_u^{(k),L}v \in \mathcal{S}_{m+n+k} (C^k L^{k-1} AB, c_2- \epsilon L^{-1} \big)$ for some $C>0$ depending only on $c_1$ and $\epsilon$. 
    \item If $0< \epsilon \ll 1$, $k \in \N_+$ and $k \le L^{\frac{1}{3}}$, then $\psi_u^{(k),L}v \in \mathcal{S}_{m+n+k} (C^k L^{\frac{2(k-1)}{3}} AB, c_2- \epsilon L^{-\frac{1}{3}} \big)$ for some $C>0$ depending only on $c_1$ and $\epsilon$.
\end{enumerate}
\end{lem}

\begin{proof}
The proof of (1) and (2) is standard. Then we prove (3). By linearity, we may assume $u$ and $v$ both consist of one term and $A=B=1$. Namely,
\begin{equation}
    u= |\alpha|^{-m-|s_1|} \alpha^{s_1} \beta^{t_1} \lambda_1^{p_1} \lambda_2^{q_1} f_{s_1,t_1,p_1,q_1}(x), \quad v= |\alpha|^{-n-|s_2|} \alpha^{s_2} \beta^{t_2} \lambda_1^{p_2} \lambda_2^{q_2} f_{s_2,t_2,p_2,q_2}(x)
\end{equation}
with
\begin{equation}
    |f_{s_1,t_1,p_1,q_1}(x)| \le e^{-c_1|x|}, \quad |f_{s_2,t_2,p_2,q_2}(x)| \le e^{-c_2|x|}.
\end{equation}

Since the Taylor series of the power function $|\alpha-\zeta|^{-1}$ has a finite radius of convergence, there is a universal constant $C>0$ such that $|F_k(\alpha, \zeta)| \le C^k |\alpha|^{-k} |\zeta|^{k-1}$ for all $k \ge 1$. We write
\begin{equation}
    F_k(\alpha, \zeta)= \sum c_{s,l} |\alpha|^{-k-|s|} \alpha^s \zeta^l,
\end{equation}
where the summation is over $s, l \in \N^3$, $|s| \le k$, $|l|= k-1$. We have $\sum |c_{s,l}| \le C^k$. 

If $k>10$, then 
\begin{equation} \begin{aligned} \label{eq explicit psi^n_u v}
    \psi_u^{(k),L} v(y_j) &= \int  f_{s_1,t_1,p_1,q_1}(\xi) F_k \big( \alpha, (-1)^j (\lambda_{j+1} \xi- \lambda_j y_j) \big) \d\xi \cdot \chi(\frac{y_j}{L}) \\
    &\qquad \ \cdot |\alpha|^{-m-n-|s_1|-|s_2|} \alpha^{s_1+s_2} \beta^{t_1+t_2} \lambda_1^{p_1+p_2} \lambda_2^{q_1+q_2} f_{s_2,t_2,p_2,q_2}(y_j) \\
    &= \sum \int f_{s_1,t_1,p_1,q_1}(\xi)  (\lambda_{j+1} \xi- \lambda_j y_j)^l \d\xi \cdot c_{s,l} \cdot (-1)^{j(k-1)} \cdot \chi(\frac{y_j}{L}) \\
    &\qquad \quad \cdot |\alpha|^{-m-n-k-|s_1|-|s_2|-|s|} \alpha^{s_1+s_2+s} \beta^{t_1+t_2} \lambda_1^{p_1+p_2} \lambda_2^{q_1+q_2} f_{s_2,t_2,p_2,q_2}(y_j).
\end{aligned} \end{equation}
We expand $(\lambda_{j+1} \xi- \lambda_j y_j)^l$ in the integrand. For any $0 \le p \le |l|=k-1$, we have 
\begin{equation}
    \int |f_{s_1,t_1,p_1,q_1}(\xi)| |\xi|^p \d\xi \le C^p p!, \quad C=C(c_1)>0,
\end{equation}
and $|y_j|^{k-p-1} \chi(\frac{y_j}{L}) \le (2L)^{k-p-1}$. It follows that $\psi_u^{(k),L} v \in \mathcal{S}_{m+n+k}(\tilde{A},c_2)$, where
\begin{equation}
    \tilde{A} \le C^k \sum_{p=0}^{k-1} p! L^{k-p-1} \le C^k (k+L)^{k-1}.
\end{equation}
Then we obtain (3) since $k \le L$.

If $k \le 10$, then there is no cutoff term $\chi(\frac{y_j}{L})$ in \eqref{eq explicit psi^n_u v}. By pairing $|y_j|^{k-p-1}$ with $e^{-\epsilon L^{-1} |y_j|}$, we deduce that $\psi_u^{(k),L} v \in \mathcal{S}_{m+n+k}(\tilde{B},c_2- \epsilon L^{-1})$, where
\begin{equation}
    \tilde{B} \le C^k \sum_{p=0}^{k-1} p! \sup \big( |y_j|^{k-p-1} e^{-\epsilon L^{-1} |y_j|} \big) \le C^k \sum_{p=0}^{k-1} p! L^{k-p-1} \sup \big( (1+|y_j|)^{10} e^{-|y_j|} \big).
\end{equation}
This proves (3) in the case $k \le 10$ as well. Thus (3) is proved.

The proof of (4) is very similar. We proceed until \eqref{eq explicit psi^n_u v} but ignore $\chi(\frac{y_j}{L})$ anyway. This time, we pair $|y_j|^{k-p-1}$ with $e^{-\epsilon L^{-\frac{1}{3}} |y_j|}$. It follows that $\psi_u^{(k),L} v \in \mathcal{S}_{m+n+k}(\tilde{C},c_2- \epsilon L^{-\frac{1}{3}})$, where
\begin{align}
    \tilde{C} &\le C^k \sum_{p=0}^{k-1} p! \sup \big( |y_j|^{k-p-1} e^{-\epsilon L^{-\frac{1}{3}} |y_j|} \big) \le C^k \sum_{p=0}^{k-1} p! L^{\frac{k-p-1}{3}} \sup \big( |y_j|^{k-p-1} e^{-|y_j|} \big) \\
    &\le C^k \sum_{p=0}^{k-1} p! L^{\frac{2(k-p-1)}{3}} \le C^k L^{\frac{2(k-1)}{3}},
\end{align}
where we have used $k \le L^{\frac{1}{3}}$ in the second line. This concludes the proof.
\end{proof}

We next show $b_j^{(n)}, m_j^{(n)}$ and $T_j^{(n)}$ are elements of the quantified admissible function spaces. We can also describe their size in terms of $n$.

\begin{lem} \label{lem quantitative description of b_j m_j T_j}
Let $b_j^{(n)}, m_j^{(n)}$ and $T_j^{(n)}$ be as defined in Proposition \ref{prop construction of approximate bubbles}. There are universal constants $C,c>0$ such that for any $L>1$, one has 
\begin{equation}
    b_j^{(n)}, m_j^{(n)} \in S_n(C^n L^{n-1}), \quad T_j^{(n)} \in \mathcal{S}_n( C^n L^{n-1},c), \quad \forall 1 \le n \le [L], \; j=1,2.
\end{equation}
\end{lem}

\begin{proof}
We take $c_0$ to be the constant in Lemma \ref{lem inverse of L+, L-}. By the proof of Lemma \ref{lem inverse of L+, L-}, we know that there is $A_0>0$ such that $Q(x) \le A_0 e^{-c_0|x|}$.

We first prove by induction that there exists $A_n>0$ for $1 \le n \le [L]$ such that 
\begin{equation} \label{eq estimate with A_n}
    m_j^{(n)}, b_j^{(n)} \in S_n(A_n), \quad T_j^{(n)} \in \mathcal{S}_n(A_n,c_0- c_0(2L)^{-1}n), \quad \forall 1 \le n \le [L], \; j=1,2.
\end{equation}
Suppose \eqref{eq estimate with A_n} holds for $1 \le n \le N$, where $N \le [L]-1$. Recall $\hat{E}_j^{(N)}$ is defined by \eqref{eq E_j^N hat}. By Lemma \ref{lem admissible functions} and \eqref{eq summation over}, we have $\hat{E}_j^{(N)} \in \mathcal{S}_{N+1}(\tilde{A}_{N+1},c_0- c_0(2L)^{-1}(N+1))$, where
\begin{align}
    \tilde{A}_{N+1} &= C_0 \sum_{\substack{k+l+m=N+1 \\ 0 \le k,l,m \le N}} A_k A_l A_m+ \sum_{\substack{k+l+m+n=N+1 \\ 0 \le k,l,m \le N,\ 1 \le n \le N+1}} C_0^n L^{n-1} A_k A_l A_m \\
    &\quad + C_0 \sum_{\substack{m+n=N+1 \\ 1 \le m,n \le N}} m A_m A_n,
\end{align}
and $C_0>0$ is a universal constant large enough. In particular, we have used $N+1 \le [L]$ and applied Lemma \ref{lem admissible functions} (3) with $c_2= c_0- c_0(2L)^{-1}N$, $\epsilon= c_0/2$. We combine the first two sums in $\tilde{A}_{N+1}$ and define a more convenient quantity 
\begin{equation} \label{eq recurrence of A_n}
    A_{N+1}= C_0 \sum_{\substack{k+l+m+n=N+1 \\ 0 \le k,l,m \le N,\ 0 \le n \le N+1}} C_0^n L^{(n-1)_+} A_k A_l A_m+ C_0 \sum_{\substack{m+n=N+1 \\ 1 \le m,n \le N}} L A_m A_n,
\end{equation} 
where $(n-1)_+= \max \{n-1,0\}$. Since $A_{N+1} \ge \tilde{A}_{N+1}$, we still have 
\begin{equation}
    \hat{E}_j^{(N)} \in \mathcal{S}_{N+1}(A_{N+1},c_0- c_0(2L)^{-1}(N+1)).
\end{equation} 
Then using \eqref{eq formula of m_j b_j}, \eqref{eq formula of T_j} and Lemma \ref{lem inverse of L+, L-}, we obtain \eqref{eq estimate with A_n} for $n=N+1$, with possibly a larger $C_0$ in \eqref{eq recurrence of A_n}. Note that \eqref{eq summation over} can be checked by counting the additional powers of $\alpha, \beta, \lambda_j$.

Therefore, we have proved \eqref{eq estimate with A_n} with $\{A_n\}$ defined by \eqref{eq recurrence of A_n}. Since $c_0- c_0(2L)^{-1}n \ge \frac{c_0}{2}$, we may take $c= \frac{c_0}{2}$ and it remains to prove 
\begin{equation}
    A_n \le C^n L^{n-1}, \quad \forall 1 \le n \le [L].
\end{equation}

Let $B_n= C_0^{-n} L^{-(n-1)_+} A_n$ for $n \ge 0$. In the first summation in \eqref{eq recurrence of A_n}, since $k,l,m,n$ cannot all be $0$, one has $(k-1)_+ +(l-1)_+ +(m-1)_+ +(n-1)_+ \le N$. Therefore, we obtain
\begin{equation}
    A_{N+1} \le C_0^{N+2} L^N \sum_{\substack{k+l+m+n=N+1 \\ 0 \le k,l,m \le N,\ 0 \le n \le N+1}} B_k B_l B_m+ C_0^{N+2} L^N \sum_{\substack{m+n=N+1 \\ 1 \le m,n \le N}} B_m B_n,
\end{equation}
and thus
\begin{equation}
    B_{N+1} \le 2C_0 \sum_{\substack{k+l+m \le N+1 \\ 0 \le k,l,m \le N}} B_k B_l B_m.
\end{equation}
We prove in Appendix \ref{appendix sequence lemma} that this implies $B_n \le C^n$, which completes the proof.
\end{proof}

\begin{rmk}
In the induction step of the above proof, if we use Lemma \ref{lem admissible functions} (4) instead of (3), then we can prove
\begin{equation} \label{eq rough estimate of data, adjusted}
    b_j^{(n)}, m_j^{(n)} \in S_n(C^n L^{\frac{2(n-1)}{3}}), \quad T_j^{(n)} \in \mathcal{S}_n( C^n L^{\frac{2(n-1)}{3}},c), \quad \forall 1 \le n \le [L^{\frac{1}{3}}], \; j=1,2.
\end{equation}
\end{rmk} 

We conclude this section with the following estimates for the accuracy of the approximate solution. It is important that the constants $c_0, C_0$ do not depend on $N$.

\begin{prop} \label{prop accuracy of approximate solution}
For $K>1$, there exist $\bar{r}_0>0$, $\eta>0$ and $c_0, C_0>0$ such that the following is true. Assume $r_0>\bar{r}_0$, $N \in \N_+$, $\frac{1}{2} \eta r_0 \le N \le L \le \eta r_0$ and $P(t)$ satisfies
\begin{equation} \label{eq rough estimate on parameters}
    |\alpha| \ge r_0, \quad |\beta| \le K, \quad K^{-1} \le \lambda_j \le K.
\end{equation}

\begin{enumerate}
    \item Let $V_j^{(N)}$, $M_j^{(N)}$ and $B_j^{(N)}$ be as in Proposition \ref{prop construction of approximate bubbles}. Then
    \begin{equation} \label{eq estimate of B_j, M_j}
        |B_j^{(N)}-b_j^{(2)}| \le C_0 r_0^{-7}, \quad |M_j^{(N)}| \le C_0 r_0^{-7},
    \end{equation}
    \begin{equation} \label{eq estimate of V_j}
        |V_j^{(N)}- Q| \le C_0 r_0^{-3} e^{-c_0 |y_j|}.
    \end{equation}
    \item Let $R_g^{(N)}$ be defined by \eqref{eq approximate solution}, and 
    \begin{equation} \label{eq definition of Psi}
        \Psi^{(N)}= i\partial_t R_g^{(N)}+ \Delta R_g^{(N)}- \phi_{|R_g^{(N)}|^2} R_g^{(N)}- \sum_{j=1}^2 \frac{1}{\lambda_j^2} S_j^{(N)} e^{i\gamma_j+ i\beta_j \cdot x}.     
    \end{equation}
    Then
    \begin{equation} \label{eq estimate of Psi}
        |\Psi^{(N)}(t,x)| \le C_0 e^{-c_0 r_0} \max \limits_j e^{-c_0 |x-\alpha_j|}.
    \end{equation} 
\end{enumerate}
\end{prop}

\begin{rmk} \label{rmk range of N, L}
We give a heuristic explanation of the range of $N$ and $L$.

The Taylor expansion of $\frac{1}{|\alpha- \zeta|}$ is valid if $|\zeta|< |\alpha|$. In \eqref{eq definition of psi^n,L}, we roughly have $|y_j| \sim |\zeta|$, so the cutoff threshold $L$ has size $r_0$.

From the last term in the expression of $\tilde{A}_{N+1}$ in the proof of Lemma \ref{lem quantitative description of b_j m_j T_j}, we see $A_n$ grows at least like $n!$, meaning that $|b_j^{(n)}| \le n! r_0^{-n}$ is the best bound one may expect for all $n$. By Stirling's formula, the minimum of $n! r_0^{-n}$ is attained near $n \sim r_0$, so the approximation order $N$ also has size $r_0$.
\end{rmk}

\begin{proof}
In the proof, we let $c_0, C_0$ denote positive constants depending only on $K$, where $c_0$ is small and $C_0$ is large. These constants may change from line to line. Recall we set $r=|\alpha|$.

(1) By ignoring $L$ completely and simply using $m_j^{(n)} \in S_n$, we have $\sum_{n=7}^{20} |m_j^{(n)}| \le C_0 r^{-7}$. Recall that from the definition of $S_n(A)$, we have
\begin{equation}
    \sigma \in S_n(A) \implies |\sigma(\alpha, \beta, \lambda_1, \lambda_2)| \le C_0^n A r^{-n} \text{ when \eqref{eq rough estimate on parameters} holds}.
\end{equation}
Then, using \eqref{eq rough estimate of data, adjusted} and $L \le \eta r_0$ with $\eta$ small enough, we get
\begin{equation}
    \sum_{21 \le n \le L^{1/3}} |m_j^{(n)}| \le \sum_{21 \le n \le L^{1/3}} C_0^n L^{\frac{2n}{3}} r^{-n} \le \sum_{21 \le n \le L^{1/3}} (C_0 \eta)^n  r^{-7} \le C_0 r^{-7}.
\end{equation}
By Lemma \ref{lem quantitative description of b_j m_j T_j} and taking $\eta$ small enough as well, we have
\begin{align}
    \sum_{L^{1/3} < n \le N} |m_j^{(n)}| \le \sum_{L^{1/3} < n \le N} C_0^n L^n r^{-n} \le (C_0 \eta)^{L^{\frac{1}{3}}} \le  e^{-L^{\frac{1}{3}}}.
\end{align}
Then we get $|M_j^{(N)}| \le C_0 r_0^{-7}$ by adding up these inequalities and using Proposition \ref{prop no lower order error}. The estimates for $B_j^{(N)}$ and $V_j^{(N)}$ are proved similarly.

(2) By \eqref{eq equation of approximate solution}, we have (omitting the superscript $N$ and the subscript $g$ for simplicity)
\begin{equation} \begin{aligned}
    \Psi &= \sum_{j=1}^2 \frac{1}{\lambda_j^2} \tilde{E}_j(t,y_j) e^{i\gamma_j+ i\beta_j \cdot x}+ 2\phi_{\Re (R_1 \ol{R_2})} R \\
    &\quad +\sum_{j=1}^2 \frac{1}{\lambda_j^2} \left[ \phi_{\left| V_{j+1} \right|^2}^{(N),L}- \frac{1}{\lambda_{j+1}^2} \phi_{|V_{j+1}|^2} \left( P(t), \frac{\lambda_j y_j}{\lambda_{j+1}}+ \frac{(-1)^j \alpha}{\lambda_{j+1}} \right) \right] V_j e^{i\gamma_j+ i\beta_j \cdot x}.
\end{aligned} \end{equation}
The second term on the first line is bounded by the right-hand side of \eqref{eq estimate of Psi} using \eqref{eq estimate of V_j}.

For the first term on the first line, using \eqref{eq definition of E_j tilde} and Lemma \ref{lem quantitative description of b_j m_j T_j}, we have
\begin{equation}
    |\tilde{E}_j| \le \sum_{n=N+1}^{4N+4} C_0^n L^n r^{-n} e^{-c_0 |y_j|} \le \sum_{n=N+1}^{4N+4} (C_0 \eta)^n e^{-c_0|y_j|} \le e^{-N} e^{-c_0|y_j|},
\end{equation}
provided $\eta$ is small enough.
Then we have $|\tilde{E}_j| \le e^{-c_0 r_0} e^{-c_0|y_j|}$ since $N \ge \frac{1}{2} \eta r_0$. 

It remains to consider the second line. We set $\zeta= (-1)^j (\lambda_{j+1} \xi- \lambda_j y_j)$. Then
\begin{align} 
    &\quad \ \phi_{\left| V_{j+1} \right|^2}^{(N),L}- \frac{1}{\lambda_{j+1}^2}\phi_{|V_{j+1}|^2} \left( \frac{\lambda_j y_j}{\lambda_{j+1}}+ \frac{(-1)^j \alpha}{\lambda_{j+1}} \right) \\
    &= \frac{1}{4\pi \lambda_{j+1}} \int |V_{j+1}(\xi)|^2 \bigg( \frac{1}{|\alpha- \zeta|}- \sum_{n=1}^N F_n(\alpha, \zeta) \bigg) \d\xi+ \sum_{n=11}^N \psi^{(n)}_{|V_{j+1}|^2} \big( 1- \chi(\frac{y_j}{L}) \big).
\end{align}
We separate the integral into $I_1+ I_2$, where 
\begin{equation}
    I_1= \int_{\{ |\zeta| \le \frac{1}{2} r \}} |V_{j+1}(\xi)|^2 \bigg( \frac{1}{|\alpha- \zeta|}- \sum_{n=1}^N F_n(\alpha, \zeta) \bigg) \d\xi
\end{equation}
and
\begin{equation}
    I_2= \int_{\{ |\zeta| \ge \frac{1}{2} r \}} |V_{j+1}(\xi)|^2 \bigg( \frac{1}{|\alpha- \zeta|}- \sum_{n=1}^N F_n(\alpha, \zeta) \bigg) \d\xi.
\end{equation}

If $|\zeta| \le \frac{1}{2} r$, then the power series $\sum F_n(\alpha, \zeta)$ converges, which implies
\begin{equation}
    \bigg| \frac{1}{|\alpha- \zeta|}- \sum_{n=1}^N F_n(\alpha, \zeta) \bigg| \le C_0^N \frac{|\zeta|^N}{r^{N+1}}.
\end{equation}
By \eqref{eq estimate of V_j}, there exists $c_1>0$ depending on $K$ such that $|V_{j+1}(\xi)| \le C_0 e^{-c_1|\xi|}$. Using 
\begin{equation}
    |\xi| \ge (2K)^{-1}|\xi| \ge \frac{1}{2} K^{-2}|\zeta|- \frac{1}{2} |y_j|
\end{equation} 
and 
\begin{equation}
    \int e^{-|x|} |x|^N \dx \le C_0^N N!
\end{equation}
we obtain
\begin{align}
    |I_1| &\le \frac{C_0^N}{r^{N+1}} \int e^{-c_1|\xi|} |\zeta|^N \d\xi \le \frac{C_0^N}{r^{N+1}} e^{\frac{1}{2} c_1 |y_j|} \int e^{-\frac{c_1}{2K^2}|\zeta|} |\zeta|^N \d\xi \le \frac{C_0^N N!}{r^{N+1}} e^{\frac{1}{2} c_1 |y_j|}.
\end{align}
Then from $\frac{1}{2} \eta r_0 \le N \le \eta r_0$, taking $\eta$ small enough, we get 
\begin{equation}
    |I_1| \le C_0 e^{-c_0 r_0} e^{\frac{1}{2} c_1 |y_j|}.
\end{equation}

If $|\zeta| \ge \frac{1}{2} r$, then we use the crude estimate
\begin{equation}
    \bigg| \frac{1}{|\alpha- \zeta|}- \sum_{n=1}^N F_n(\alpha, \zeta) \bigg| \le \frac{1}{|\alpha- \zeta|}+C_0^N \frac{|2\zeta|^N}{r^N}.
\end{equation}
The contribution of the second term is estimated as $I_1$. For the first term $\frac{1}{|\alpha- \zeta|}$, using
\begin{equation}
    |\xi| \ge \frac{1}{2} K^{-2}|\zeta|- \frac{1}{2} |y_j| \ge \frac{1}{8} K^{-2}r+ \frac{1}{4} K^{-2}|\zeta|- \frac{1}{2} |y_j|
\end{equation}
and integrability of $e^{-|x|} \frac{1}{|\alpha-x|}$, we have
\begin{equation}
    \int_{\{ |\zeta| \ge \frac{1}{2} r \}} |V_{j+1}(\xi)|^2 \frac{1}{|\alpha- \zeta|} \d\xi \le e^{-\frac{c_1}{8K^2} r} e^{\frac{1}{2}c_1 |y_j|} \int e^{-\frac{c_1}{4K^2} |\zeta|} \frac{1}{|\alpha-\zeta|} \d\xi \le C_0 e^{-c_0 r} e^{\frac{1}{2} c_1 |y_j|}. 
\end{equation}
We thus derive
\begin{equation}
    |I_2| \le C_0 e^{-c_0 r_0} e^{\frac{1}{2} c_1 |y_j|}.
\end{equation}

Finally, since $|F_n(\alpha, \zeta)| \le C^n |\alpha|^{-n} |\zeta|^{n-1}$, by similar estimates as for $I_1$, we have
\begin{equation}
    \big| \psi^{(n)}_ {|V_{j+1}|^2} (y_j) \big| \le \frac{C_0^n n!}{r^n} e^{\frac{1}{2} c_1 |y_j|} \le e^{-n} e^{\frac{1}{2} c_1 |y_j|} \;\; \text{for} \;\; n \le N,
\end{equation}
provided $\eta$ is small enough. Then using $L \ge \frac{1}{2} \eta r_0$, we obtain 
\begin{equation}
    \sum_{n=11}^N \Big| \psi^{(n)}_{|V_{j+1}|^2} \big( 1- \chi(\frac{y_j}{L}) \big) \Big| \le \big( 1- \chi(\frac{y_j}{L}) \big) e^{\frac{1}{2} c_1 |y_j|} \le e^{-c_0r_0} e^{\frac{2}{3} c_1 |y_j|}.
\end{equation}

With all the above estimates as well as $|V_j(y_j)| \le C_0 e^{-c_1|y_j|}$, we deduce \eqref{eq estimate of Psi}.
\end{proof}

\section{Stability of the approximate center dynamics} \label{sec ode analysis}

In this section, we study the $N$-th order approximation of the $2$-body problem 
\begin{equation} \label{eq 2-body N-th}
    \dot{\alpha}_j= 2\beta_j, \quad \dot{\beta}_j= B_j^{(N)}, \quad \dot{\lambda}_j= M_j^{(N)}, \quad j=1,2.
\end{equation}
where $B_j^{(N)}$ and $M_j^{(N)}$ are as in Proposition \ref{prop construction of approximate bubbles}. If we can find a solution to this, then after solving $\gamma_j$ from
\begin{equation}
    \dot{\gamma}_j= \frac{1}{\lambda_j^2}- |\beta_j|^2- \dot{\beta}_j \cdot \alpha_j+ \frac{\| Q \|_{L^2}^2}{4\pi \lambda_{j+1}} \frac{1}{|\alpha|},
\end{equation}
we will be able to eliminate the modulation error $S_j^{(N)}$.

Our goal is to show that on a certain time scale, the solution to \eqref{eq 2-body N-th} remains close to the solution to \eqref{eq 2-body}. Note that Proposition \ref{prop no lower order error} is an essential input for stability as it implies that the approximate acceleration differs from the Kepler acceleration only at order $|\alpha|^{-7}$. For this reason, we provide the proof of Proposition \ref{prop no lower order error} now.

First, since we only care about $n \le 6$, by \eqref{eq definition of psi^n,L}, we always have $\psi^{(n),L}= \psi^{(n)}$ and the cutoff function $\chi$ is not involved.

Before carrying out the calculation, we need some preliminary facts. As in \cite[Lemma~4.2]{Hartree4Dmultisoliton}, we have $\Delta_x F_n(\alpha,x)= 0$. This is because $F_n$ comes from the Taylor expansion of a harmonic function. This leads us to consider spherical harmonics; for reference, see \cite{Harmonictextbook}.

For $l \ge 0$, let $\mathcal{Y}_l= \left\{ P|_{S^2} \ \big| \ P \text{ is homogeneous of degree } l \text{ and } \Delta P=0 \right\}$ and 
\begin{equation}
    \H_l= L^2 \big( (0,+\infty), r^2 \d r \big) \otimes \mathcal{Y}_l,
\end{equation}
which is the \emph{$l$-th spherical harmonic subspace of $L^2(\R^3)$}.
The following properties will be useful:
\begin{enumerate} [label=(\alph*)]
    \item \label{item orthogonal} $L^2(\R^3)= \bigoplus_{l=0}^\infty \H_l$, and $\H_m \perp \H_n$ whenever $m \neq n$.
    \item \label{item product rule} If $m<n$, then $\H_m \cdot \H_n \cap L^2 \subset \H_{n-m} \oplus \H_{n-m+2} \oplus \cdots \oplus \H_{m+n}$.
\end{enumerate}
Property \ref{item orthogonal} implies that components in $\H_l,\ l \ge 2$ do not affect $b_j^{(n)}$ and $m_j^{(n)}$ since $\nabla Q \in \H_1$ and $Q \in \H_0$. It also gives some useful cancellation in our calculation. Property \ref{item product rule} will help us project terms in \eqref{eq E_j^N hat} into the spaces $\H_l$.  

We can give more information on $\psi^{(n)}$. Assume $f \in \H_l$ and $f$ decays exponentially. Recall
\begin{equation}
    \psi_{f,j+1}^{(n)} (y_j)= -\frac{1}{4\pi \lambda_{j+1}} \int f(\xi) F_n \big( \alpha, (-1)^j (\lambda_{j+1} \xi- \lambda_j y_j) \big) \d \xi.
\end{equation}
Note $F_n \in \H_{n-1}$ because $F_n \not\in L^2$. But $F_n \in \tilde{\H}_{n-1}$ for $n \ge 1$, where
\begin{equation}
    \tilde{\H}_l= \left\{ f:\R^3 \to \C \ \big| \ \forall \epsilon>0, e^{-\epsilon|\cdot|} f(\cdot) \in \H_l \right\}.
\end{equation}
Take $0<\epsilon \ll 1$ so that $e^{\epsilon|\xi|} f(\xi)$ is still in $\H_l$. For each $y_j$, we have
\begin{equation}
    e^{-\epsilon|\xi|} F_n \big( \alpha, (-1)^j (\lambda_{j+1} \xi- \lambda_j y_j) \big) \in \H_0 \oplus \cdots \oplus \H_{n-1} \quad (\text{in } \xi),
\end{equation} 
and the $\H_k$-component is homogeneous of degree $n-k-1$ in $y_j$. By Property \ref{item orthogonal}, we deduce 
\begin{equation} \label{eq psi^n regarding H_l}
    \psi_{f,j+1}^{(n)}(y_j) \left\{ \begin{aligned}
        &=0, && n \le l, \\
        &\in \tilde{\H}_{n-l-1}, && n>l.
    \end{aligned} \right.
\end{equation}
In particular, we have $\psi_{Q^2,j+1}^{(N+1)} \in \tilde{\H}_N$ and thus $\psi_{Q^2,j+1}^{(N+1)} Q \in \H_N$. Then Property \ref{item orthogonal} gives
\begin{equation} \label{eq psi orthogonality}
    \big( \psi^{(N+1)}_{Q^2,j+1} Q, \nabla Q \big)=0, \quad \forall N \ge 2.
\end{equation}

Finally, note that $L_\pm: D(L_\pm) \cap \H_l \to \H_l$. In Lemma \ref{lem inverse of L+, L-}, if we have $f \in \H_l$, then the inverse $u$ can also be chosen in $\H_l$. We will use this choice when solving $T_j^{(n)}$ from \eqref{eq formula of T_j}.

\begin{proof}[Proof of Proposition \ref{prop no lower order error}]
With the convention that $V_j^{(0)}= T_j^{(0)}=Q$, we can define $\tilde{E}_j^{(N)}$ for $N=0$. Then \eqref{eq E_j^N hat}, \eqref{eq formula of m_j b_j} and \eqref{eq formula of T_j} are valid for all $N \ge 0$.

\;\\
\textbf{1st order:} Using \eqref{eq E_j^N hat}, we have 
\begin{equation}
    \hat{E}_j^{(0)}= -\frac{\| Q \|_{L^2}^2}{4\pi \lambda_{j+1}} \frac{1}{|\alpha|} Q.
\end{equation} 
We compute directly $\psi_{Q^2,j+1}^{(1)}= -\frac{\| Q \|_{L^2}^2}{4\pi \lambda_{j+1}} \frac{1}{|\alpha|}$. Then using \eqref{eq formula of m_j b_j}, we get 
\begin{equation}
    b_j^{(1)}=0, \quad m_j^{(1)}=0 
\end{equation}
and thus \eqref{eq formula of T_j} yields
\begin{equation}
    T_j^{(1)}=0.
\end{equation} 
We emphasize that this is the place where the interaction shift $-\frac{\| Q \|_{L^2}^2}{4\pi \lambda_{j+1}} \frac{1}{|\alpha|} V_j^{(N)}$ in \eqref{eq definition of E_j tilde} is useful. It cancels $\psi_{Q^2,j+1}^{(1)} Q$ when $N=0$, which yields $T_j^{(1)}=0$.

\;\\
\textbf{2nd order:} Then we have $\hat{E}_j^{(1)}= 0$. We compute directly
\begin{equation}
    \psi_{Q^2,j+1}^{(2)} = (-1)^j\frac{\| Q \|_{L^2}^2 \lambda_j}{4\pi \lambda_{j+1}} \frac{\alpha \cdot y_j}{|\alpha|^3}.
\end{equation}
Then using \eqref{eq formula of m_j b_j}, we get
\begin{equation}
    b_j^{(2)}= \frac{(-1)^{j-1}\| Q \|_{L^2}^2}{4\pi \lambda_{j+1}} \frac{\alpha}{|\alpha|^3}, \quad m_j^{(2)}=0
\end{equation}
and thus \eqref{eq formula of T_j} yields
\begin{equation}
    T_j^{(2)}=0.
\end{equation}

For higher orders, we do not compute the explicit formula of $T_j^{(n)}$. Instead, we use properties of spherical harmonics to determine which $\H_l$-component of $T_j^{(n)}$ is nonzero. This information will already be enough for us to obtain $b_j^{(n)}$ and $m_j^{(n)}$ for $1 \le n \le 6$. 

The following expression will be used repeatedly:
\begin{equation} \label{eq psi^1 and psi^2}
    \psi^{(1)}_{Q^2,j+1}= -\frac{\| Q \|_{L^2}^2}{4\pi \lambda_{j+1}} \frac{1}{|\alpha|}, \quad \psi^{(2)}_{Q^2, j+1}= -\lambda_j b_j^{(2)} \cdot y_j.
\end{equation}

\;\\
\textbf{3rd order:} Then we have $\hat{E}_j^{(2)}=0$. Using \eqref{eq formula of m_j b_j} and \eqref{eq psi orthogonality}, we have
\begin{equation}
    b_j^{(3)}=0, \quad m_j^{(3)}=0.
\end{equation}
Then \eqref{eq formula of T_j} and \eqref{eq psi^n regarding H_l} yield
\begin{equation}
    \Re \ T_j^{(3)} \in \H_2, \quad \Im \ T_j^{(3)}=0.
\end{equation}

\;\\
\textbf{4th order:} Then we have 
\begin{equation}
    \hat{E}_j^{(3)}= -2\psi^{(1)}_{QT_{j+1}^{(3)}} Q- \psi^{(1)}_{Q^2,j+1} T_j^{(3)}- \frac{\|Q\|_{L^2}^2}{4\pi \lambda_{j+1}} \frac{1}{|\alpha|} T_j^{(3)}+ i\frac{\partial T_j^{(3)}}{\partial \alpha} \cdot 2\beta.
\end{equation}
We have $\psi^{(1)}_{Q T_{j+1}^{(3)}}=0$ by \eqref{eq psi^n regarding H_l} since $QT_{j+1}^{(3)} \in \H_2$. Then by \eqref{eq psi^1 and psi^2}, we get $\hat{E}_j^{(3)}= i\frac{\partial T_j^{(3)}}{\partial \alpha} \cdot 2\beta$. Therefore, 
\begin{equation}
    \Re \ \hat{E}_j^{(3)}=0, \quad \Im \ \hat{E}_j^{(3)} \in \H_2.
\end{equation}
Using \eqref{eq formula of m_j b_j} and \eqref{eq psi orthogonality}, we have
\begin{equation}
    b_j^{(4)}=0, \quad m_j^{(4)}=0.
\end{equation}
Then \eqref{eq formula of T_j} and \eqref{eq psi^n regarding H_l} yield
\begin{equation}
    \Re \ T_j^{(4)} \in \H_3, \quad \Im \ T_j^{(4)} \in \H_2.
\end{equation}

\;\\
\textbf{5th order:} Then we have
\begin{align}
    \hat{E}_j^{(4)}= & -2\psi^{(1)}_{Q \Re \ T_{j+1}^{(4)}} Q- \psi^{(1)}_{Q^2,j+1} T_j^{(4)}- 2\psi^{(2)}_{Q T_{j+1}^{(3)}} Q- \psi^{(2)}_{Q^2, j+1} T_j^{(3)} \\
    &- \lambda_j b_j^{(2)} \cdot y_j T_j^{(3)}- \frac{\| Q \|_{L^2}^2}{4\pi \lambda_{j+1}} \frac{1}{|\alpha|} T_j^{(4)}+ i\frac{\partial T_j^{(4)}}{\partial \alpha} \cdot 2\beta.
\end{align}
We have $\psi^{(1)}_{Q \Re \ T_{j+1}^{(4)}}= \psi^{(2)}_{QT_{j+1}^{(3)}}= 0$ by \eqref{eq psi^n regarding H_l}. Then by \eqref{eq psi^1 and psi^2}, we get $\hat{E}_j^{(4)}= i\frac{\partial T_j^{(4)}}{\partial \alpha} \cdot 2\beta$. Therefore, 
\begin{equation}
    \Re \ \hat{E}_j^{(4)} \in \H_2, \quad \Im \ \hat{E}_j^{(4)} \in \H_3.
\end{equation}
Using \eqref{eq formula of m_j b_j} and \eqref{eq psi orthogonality}, we have
\begin{equation}
    b_j^{(5)}=0, \quad m_j^{(5)}=0.
\end{equation}
Then \eqref{eq formula of T_j} and \eqref{eq psi^n regarding H_l} yield
\begin{equation}
    \Re \ T_j^{(5)} \in \H_2 \oplus \H_4, \quad \Im \ T_j^{(5)} \in \H_3.
\end{equation}

\;\\
\textbf{6th order:} Then we have
\begin{align}
    \hat{E}_j^{(5)}= &-\lambda_j^{-2} \phi_{|T_j^{(3)}|^2} Q- 2\lambda_j^{-2} \phi_{Q T_j^{(3)}} T_j^{(3)}- 2\psi^{(1)}_{Q \Re \ T_{j+1}^{(5)}} Q- \psi^{(1)}_{Q^2,j+1} T_j^{(5)} \\
    &- \psi^{(2)}_{Q^2,j+1} T_j^{(4)}- 2\psi^{(2)}_{Q \Re \ T_{j+1}^{(4)}} Q- \psi^{(3)}_{Q^2,j+1} T_j^{(3)}- 2\psi^{(3)}_{Q T_{j+1}^{(3)}} Q \\
    &- \lambda_j b_j^{(2)} \cdot y_j T_j^{(4)}- \frac{\| Q \|_{L^2}^2}{4\pi \lambda_{j+1}} \frac{1}{|\alpha|} T_j^{(5)}+ i\frac{\partial T_j^{(5)}}{\partial \alpha} \cdot 2\beta+ i\frac{\partial T_j^{(4)}}{\partial \beta} \big( b_2^{(2)}- b_1^{(2)} \big).
\end{align}
We have $\psi^{(1)}_{Q \Re \ T_{j+1}^{(5)}}= \psi^{(2)}_{Q \Re \ T_{j+1}^{(4)}}=0$ by \eqref{eq psi^n regarding H_l}. Then by \eqref{eq psi^1 and psi^2}, we get
\begin{align}
    \hat{E}_j^{(5)}= &-\lambda_j^{-2} \phi_{|T_j^{(3)}|^2} Q- 2\lambda_j^{-2} \phi_{Q T_j^{(3)}} T_j^{(3)}- \psi^{(3)}_{Q^2,j+1} T_j^{(3)}- 2\psi^{(3)}_{Q T_{j+1}^{(3)}} Q \\
    &+ i\frac{\partial T_j^{(5)}}{\partial \alpha} \cdot 2\beta+ i\frac{\partial T_j^{(4)}}{\partial \beta} \big( b_2^{(2)}- b_1^{(2)} \big).    
\end{align}
The first line, consisting of terms with $T_j^{(3)}$, is real-valued. Using Property \ref{item product rule} and \eqref{eq psi^n regarding H_l}, we see these terms are in $\H_0 \oplus \H_2 \oplus \H_4 \oplus \H_6$. Therefore,
\begin{equation}
    \Re \ \hat{E}_j^{(5)} \in \H_0 \oplus \H_2 \oplus \H_3 \oplus \H_4 \oplus \H_6, \quad \Im \ \hat{E}_j^{(5)} \in \H_2 \oplus \H_3 \oplus \H_4.
\end{equation}
Using \eqref{eq formula of m_j b_j} and \eqref{eq psi orthogonality}, we have
\begin{equation}
    b_j^{(6)}=0, \quad m_j^{(6)}=0.
\end{equation}
We have thus proved Proposition \ref{prop no lower order error}.
\end{proof}

\begin{rmk}
With some additional work, one can show that $b_j^{(7)}=0$ if and only if $\lambda_1= \lambda_2$. Therefore, in the setting of this paper where we do not make the equal-mass assumption, Proposition \ref{prop no lower order error} is optimal. 

This also shows that extra cancellation occurs when $\lambda_1= \lambda_2$ and one may be able to obtain more precise dynamical information under this assumption, though it is not needed here. In fact, Proposition \ref{prop stability of 2-body} below already gives a longer time scale for the stability of the center dynamics than the time scale we will obtain later in Section \ref{sec error estimates} for energy estimates.
\end{rmk}

We now prove the following stability result. Recall our notation that $\alpha= \alpha_2- \alpha_1$, $\beta= \beta_2- \beta_1$ and $\lambda= (\lambda_1, \lambda_2)$.

\begin{prop} \label{prop stability of 2-body}
Let $C_e>0$ and $\lambda_1^0, \lambda_2^0>0$. Then there exist $\eta>0$, $C>0$ and $\bar{r}_0>0$ such that for any $r_0> \bar{r}_0$ and $\frac{1}{2} \eta r_0 \le N \le L \le \eta r_0$, if $P^0(t)$ is an elliptic solution to \eqref{eq 2-body} satisfying 
\begin{equation}
    r_0 \le |\alpha^0(t)| \le 
    C_e r_0, \quad \forall t \ge 0,
\end{equation}
then the solution $P^{(N)}$ to \eqref{eq 2-body N-th} with initial data $P^{(N)}(0)=P^0(0)$ satisfies
\begin{equation} 
    \left\{\begin{aligned} &|\alpha^{(N)}(t)- \alpha^0(t)| \le C r_0^{-3/2}, \\
    &|\beta^{(N)}(t)-\beta^0(t)| \le C r_0^{-3}, \\
    &|\lambda^{(N)}(t)- \lambda^0| \le C r_0^{-9/2}, 
    \end{aligned} \right. \quad \forall 0 \le t \le r_0^{5/2}.
\end{equation}
\end{prop}

\begin{rmk}
The time scale $r_0^{5/2}$ can be enlarged, at the cost of a weaker upper bound. However, since the time scale in the theorem is $r_0^2$, this is sufficient for out purposes.
\end{rmk}

\begin{proof}
By a bootstrap argument, it suffices to prove the following: for any $T_* \in [0,r_0^{5/2}]$, if 
\begin{equation} \label{eq ode bootstrap assumption}
    \left\{\begin{aligned} &|\alpha^{(N)}(t)- \alpha^0(t)| \le C r_0^{-3/2}, \\
    &|\beta^{(N)}(t)-\beta^0(t)| \le C r_0^{-3}, \\
    &|\lambda^{(N)}(t)- \lambda^0| \le C r_0^{-9/2}, 
    \end{aligned} \right. \quad \forall 0 \le t \le T_*,
\end{equation}
then
\begin{equation} 
    \left\{\begin{aligned} &|\alpha^{(N)}(t)- \alpha^0(t)| \le \frac{1}{2} C r_0^{-3/2}, \\
    &|\beta^{(N)}(t)-\beta^0(t)| \le \frac{1}{2} C r_0^{-3}, \\
    &|\lambda^{(N)}(t)- \lambda^0| \le \frac{1}{2} C r_0^{-9/2}, 
    \end{aligned} \right. \quad \forall 0 \le t \le T_*.
\end{equation}
Then \eqref{eq ode bootstrap assumption} implies $P^{(N)}$ satisfies \eqref{eq rough estimate on parameters}, and thus \eqref{eq estimate of B_j, M_j} holds since $\frac{1}{2} \eta r_0 \le N \le L \le \eta r_0$.

The third line is straightforward by integrating the equation for $\lambda$:
\begin{equation}
    |\lambda^{(N)}(t)- \lambda^0| \le \int_0^t |M_1^{(N)}|+ |M_2^{(N)}| \le C_0 r_0^{-7} t \le C_0 r_0^{-9/2}, \quad \forall t \le T_*,
\end{equation}
so taking $C$ large enough will suffice. It remains to deal with the first two lines.

For an elliptic trajectory of size $r_0$, the relative velocity has size $r_0^{-1/2}$, and the fundamental period has size $r_0^{3/2}$. Thus we normalize the problem by scaling
\begin{equation}
    \tau= r_0^{-3/2} t, \quad x^0= r_0^{-1} \alpha^0, \quad y^0= r_0^{1/2} \beta^0, \quad x= r_0^{-1} \alpha^{(N)}, \quad y= r_0^{1/2} \beta^{(N)}.
\end{equation}
We have $1 \le |x^0| \le C_e$ and $0 \le \tau \le r_0^{-3/2} T_* \le r_0$. Moreover, \eqref{eq 2-body} implies
\begin{equation} \label{eq 2-body dif}
    \dot{x}^0= 2y^0, \quad \dot{y}^0= -m_0 \frac{x^0}{|x^0|^3},
\end{equation} 
and \eqref{eq 2-body N-th} implies
\begin{equation} \label{eq 2-body N-th dif}
    \dot{x}= 2y, \quad \dot{y}= -m_0 \frac{x}{|x|^3}+ F,
\end{equation}
where a dot denotes differentiation with respect to $\tau$, $m_0= \frac{\|Q\|_{L^2}^2}{4\pi} (\frac{1}{\lambda_1^0}+ \frac{1}{\lambda_2^0}) $ and
\begin{align}
    F= & \ r_0^2 \sum_{j=1}^2 (-1)^j\Big( B_j^{(N)}(\alpha^{(N)}, \beta^{(N)}, \lambda^{(N)})- b_j^{(2)}(\alpha^{(N)}, \beta^{(N)}, \lambda^{(N)}) \Big) \\
    &- \frac{\|Q\|_{L^2}^2}{4\pi} \Big( \frac{1}{\lambda_1^{(N)}}- \frac{1}{\lambda_1^0}+ \frac{1}{\lambda_2^{(N)}}- \frac{1}{\lambda_2^0} \Big) \frac{x}{|x|^3} .
\end{align}
By \eqref{eq ode bootstrap assumption} and \eqref{eq estimate of B_j, M_j}, we have
\begin{equation} \label{eq estimate of F}
    |F| \le C_0 r_0^2 r_0^{-7}+ C_0 r_0^{-9/2} \le C_0 r_0^{-9/2}.
\end{equation}

Let us write $U^0=(x^0,y^0)$ and $U=(x,y)$ so that \eqref{eq 2-body dif} and \eqref{eq 2-body N-th dif} can be viewed as $6$-dimensional first-order ODEs of $U^0$ and $U$, respectively. Let $\Phi_\tau$ denote the flow associated with \eqref{eq 2-body dif}. Namely, for $Z \in \R^6$, let $\Phi_\tau(Z)$ be the solution of \eqref{eq 2-body dif} with initial data $\Phi_0(Z)=Z$. Strictly speaking, $\Phi_\tau$ is only defined on a subset of $\R^6$.

\begin{lem}
We have the following Duhamel's formula
\begin{equation} \label{eq duhamel}
    U(\tau)- U^0(\tau)= \int_0^\tau D\Phi_{\tau-s} (U(s))(0,F(s)) \d s, \quad \forall \tau \ge 0.
\end{equation}
\end{lem}

\begin{proof}
Let $W(s)=\Phi_{\tau-s}(U(s))$ and $V(X,Y)= (2Y, -m_0 \frac{X}{|X|^3})$. Then 
\begin{equation}
    \frac{\d}{\d s} W(s)= - V(\Phi_{\tau-s} (U(s)))+ D\Phi_{\tau-s} (U(s)) U'(s).
\end{equation}
From the flow property $\Phi_{s+h}(Z)=\Phi_s \Phi_h(Z)$, taking the derivative in $h$ and then setting $h=0$ gives $V(\Phi_s(Z))= D\Phi_s(Z) V(Z)$.
Then replacing $s$ by $\tau-s$ and $Z$ by $U(s)$, we obtain
\begin{equation}
    V(\Phi_{\tau-s} (U(s)))= D\Phi_{\tau-s} (U(s)) V(U(s)).
\end{equation}
Note $U'(s)= V(U(s))+ (0,F(s))$. We thus deduce
\begin{equation}
    \frac{\d}{\d s} W(s)= D\Phi_{\tau-s} (U(s))(0,F(s)).
\end{equation}
Then \eqref{eq duhamel} follows since $W(0)= U^0(\tau)$ and $W(\tau)= U(\tau)$.
\end{proof}

Next, we show that $D\Phi_\tau$ grows at most linearly in $\tau$.

\begin{lem} \label{lem Hill estimate}
Fix $0<\epsilon \ll 1$. Let $K_\epsilon$ be the set of $Z \in \R^6$ such that $X(\tau)$ describes an ellipse and $\epsilon \le |X(\tau)| \le \epsilon^{-1}$ when we write $\Phi_\tau(Z)= (X(\tau), Y(\tau))$. Then 
\begin{equation}
    \| D\Phi_\tau(Z) \| \lesssim_\epsilon 1+\tau, \quad \forall Z \in K_\epsilon, \; \tau \ge 0.
\end{equation}
\end{lem}

\begin{proof}
We write $K=K_\epsilon$ and let $C$ denote constants depending only on $\epsilon$. 

We first show that $K$ is compact. For $Z=(X_0, Y_0) \in K$, we have $\epsilon \le |X_0| \le \epsilon^{-1}$ and 
\begin{equation}
    |Y_0|^2- \frac{m_0}{|X_0|}<0
\end{equation}
since $X(\tau)$ describes an ellipse. Thus $|Y_0|^2 \le 2m_0 \epsilon^{-1}$, so $K$ is bounded. Since $X(\tau)$ is continuous in $Z$ and the lower bound of $|X(\tau)|$ ensures that there is no blow-up, we see that $K$ is closed. Therefore, $K$ is a compact set.

For $Z \in K$, let $a(Z)$ be the semi-major axis of $X(\tau)$ and $T(Z)$ be the fundamental period. By the smooth dependence of parameters and compactness of $K$, we have
\begin{equation}
    |DT(Z)| \le C, \quad \forall Z \in K.
\end{equation}
Moreover, $T(Z) /a(Z)^3$ is constant by Kepler's third law. Since $\epsilon \le |X(\tau)| \le \epsilon^{-1}$, we have $C^{-1} \le a(Z) \le C$ and thus $C^{-1} \le T(Z) \le C, \; \forall Z \in K$. Then by the smoothness of $\Phi_\tau(Z)$ in both $\tau$ and $Z$, we have
\begin{equation}
    \| D\Phi_{\tau}(Z) \| \le C, \quad \forall Z \in K, \; \tau \in [0,T(Z)].
\end{equation}

Let $n \in \N$. For any $\tau \ge 0$, let $s(Z)=\tau- n T(Z)$. Then 
\begin{equation}
    \Phi_\tau(Z)= \Phi_{s(Z)}(Z).
\end{equation}
We differentiate both sides in $Z$. Since $Ds(Z)= -n DT(Z)$, we get
\begin{equation}
    D\Phi_\tau(Z)= D\Phi_{s(Z)}(Z)- nV(\Phi_{s(Z)}(Z)) DT(Z).
\end{equation}
Therefore, if $s(Z) \in [0,T(Z)]$, or equivalently, $\tau \in [nT(Z), (n+1)T(Z)]$, then
\begin{equation}
    \| D\Phi_\tau(Z) \| \le C+ Cn.
\end{equation}
This implies the desired estimate since $n \le \tau/T(Z)$ and $T(Z) \ge C^{-1}$.
\end{proof}

The estimate now follows from a straightforward calculation. By \eqref{eq ode bootstrap assumption}, if $r_0$ is large enough, then $U(s)$ remains close to $U^0(s)$, and thus $U(s) \in K_\epsilon$ for some $\epsilon$ depending on $C_e$. We can then apply Lemma \ref{lem Hill estimate} to get 
\begin{equation}
    \| D\Phi_{\tau-s} (U(s)) \| \le C_0 (1+\tau).
\end{equation}
Note that $0 \le t \le T_*$ implies $0 \le \tau \le r_0$. Using \eqref{eq duhamel} and \eqref{eq estimate of F}, we obtain
\begin{equation}
    |U(\tau)- U^0(\tau)| \le C_0 \tau (1+\tau) r_0^{-9/2} \le C_0 r_0^{-5/2}, \quad \forall 0 \le t \le T_*.
\end{equation}
Finally, we scale back to the original parameters and we get
\begin{equation}
    |\alpha^{(N)}(t)- \alpha^0(t)| \le C_0 r_0^{-3/2}, \quad |\beta^{(N)}(t)-\beta^0(t)| \le C_0 r_0^{-3}, \quad \forall 0 \le t \le T_*.
\end{equation}
Then the desired estimate follows by taking $C$ large enough.
\end{proof}

\section{Modulation and error estimates} \label{sec error estimates}

In this section, we prove the main theorem. We first reduce Theorem \ref{thm finite time stability} to a statement about the approximate solution. 

\begin{prop} \label{prop finite time stability with N}
Given $K>1$ and $\delta>0$, there exist $\nu>0$, $\eta>0$, $\bar{r}_0>0$ such that for any $r_0> \bar{r}_0$, there exists $\kappa>0$ with the following property. 

If $\frac{1}{2} \eta r_0 \le N \le L \le \eta r_0$ and $P^{(N)}(t)$ is a solution to \eqref{eq 2-body N-th} with 
\begin{equation} \label{eq estimate of P^N}
    |\alpha^{(N)}(t)| \ge r_0, \quad |\beta^{(N)}(t)| \le K, \quad K^{-1} \le \lambda_j^{(N)}(t) \le K, \quad \forall 0 \le t \le \nu r_0^2
\end{equation}
and $\gamma_1^{(N)}(t)$, $\gamma_2^{(N)}(t)$ satisfy
\begin{equation}
    \dot{\gamma}_j^{(N)}= \frac{1}{(\lambda_j^{(N)})^2}- |\beta_j^{(N)}|^2- \dot{\beta}_j^{(N)} \cdot \alpha_j^{(N)}+ \frac{\| Q \|_{L^2}^2}{4\pi \lambda_{j+1}^{(N)}} \frac{1}{|\alpha^{(N)}|}, \quad j=1,2,
\end{equation}
then for $u_0 \in H^1(\R^3)$,
\begin{equation} \label{eq closeness of initial data}
    \Big\| u_0- R_{g^{(N)}}^{(N)}(0, \cdot) \Big\|_{H^1}< \kappa
\end{equation}
implies
\begin{equation} \label{eq finite time stability}
    \Big\| u(t, \cdot)- (g_1^{(N)} Q)(t, \cdot)- (g_2^{(N)} Q)(t,\cdot) \Big\|_{H^1}< \delta, \quad \forall 0 \le t \le \nu r_0^2,
\end{equation}
where $u$ is the solution to \eqref{eq hartree} with $u(0)= u_0$.
\end{prop}

\begin{proof}[Proof of Theorem \ref{thm finite time stability} assuming Proposition \ref{prop finite time stability with N}] \ 

For an elliptic solution $P^0$ to \eqref{eq 2-body} as in Theorem \ref{thm finite time stability}, take $N$, $L$ such that $\frac{1}{2} \eta r_0 \le N \le L \le \eta r_0$ and let $P^{(N)}$ be the solution to \eqref{eq 2-body N-th} with the same initial data. Let $\gamma_j^{(N)}$ also have the same initial data as $\gamma_j$.

We take $R_0= R_{g^{(N)}}^{(N)}(0, \cdot)$, which satisfies \eqref{eq R_0 close to initial} because of \eqref{eq estimate of V_j}. Proposition \ref{prop stability of 2-body} implies \eqref{eq estimate of P^N} for $K$ and $r_0$ large enough, so we may apply Proposition \ref{prop finite time stability with N} to get \eqref{eq finite time stability}. 

Finally, we deduce \eqref{eq estimate in thm} from \eqref{eq finite time stability} by using Proposition \ref{prop stability of 2-body} and a similar estimate regarding the phase $\gamma_j^{(N)}$ for $r_0$ large enough.
\end{proof}

Now it remains to prove Proposition \ref{prop finite time stability with N}. The strategy is similar to the proof of \cite[Proposition~3.3]{KMR2bodyHartree} or \cite[Proposition~3.2]{Hartree3Dmultisoliton}, especially Step 1--3 below, except for the Gr\"onwall argument. The main difference is that since we are studying the elliptic case, we do not gain any decay by taking $t$ large enough, as opposed to \cite{KMR2bodyHartree, Hartree3Dmultisoliton}. Instead, the final estimate relies essentially on the uniform smallness of the approximation error obtained in Proposition \ref{prop accuracy of approximate solution}. The choice $N \sim r_0$ ensures that the constants $c_0, C_0$ there do not depend on $N$.

\begin{proof}[Proof of Proposition \ref{prop finite time stability with N}] 
In the proof, we will let $c_0, C_0$ denote positive constants that depend only on $K$. We write $f=O_0(|g|)$ if $|f| \le C_0 |g|$.

\;\\
\textbf{Step 1:}
Let $u$ be the solution to \eqref{eq hartree} with $u(0,\cdot)=u_0$ satisfying \eqref{eq closeness of initial data}. Using a bootstrap argument and the inverse function theorem, it is standard that for $r_0$ large enough and $\kappa$ small enough, there exists $g(t)$ such that if 
\begin{equation}
    \varepsilon(t,x)= u(t,x)- R_g^{(N)}(t,x),
\end{equation}
then $\| \varepsilon \|_{H^1}$ is small, and for $t \in [0,\nu r_0^2]$ and $j=1,2$, we have
\begin{equation} \label{eq orthogonality} \begin{aligned}
    &\mathrm{Re} \Big( \varepsilon(t), g_j V_j^{(N)} \Big)= \mathrm{Re} \Big( \varepsilon(t), g_j \big( y_j V_j^{(N)} \big) \Big) \\
    = &\ \mathrm{Im} \Big( \varepsilon(t), g_j \big( \Lambda V_j^{(N)} \big) \Big)= \mathrm{Im} \Big( \varepsilon(t), g_j \big( \nabla V_j^{(N)}\big) \Big)= 0.    
\end{aligned} \end{equation}
In particular, we have
\begin{equation} 
    \| g(0)- g^{(N)}(0) \|< C_0 \kappa,  \quad \| \varepsilon(0) \|_{H^1}< C_0 \kappa.
\end{equation}
A detailed proof of this statement can be found in \cite[Lemma~3.3]{Hartree3Dmultisoliton}.

To simplify notation, we write $R_j= R_{j,g}^{(N)}$, $R= R_g^{(N)}$, $M_j= M_j^{(N)}$, $B_j= B_j^{(N)}$, $V_j= V_j^{(N)}$, $S_j= S_j^{(N)}$ and $\Psi= \Psi^{(N)}$. We have
\begin{equation} \label{eq equation of epsilon}
    i\partial_t \varepsilon+ \Delta \varepsilon- 2\phi_{\Re (\varepsilon \ol{R})}R- \phi_{|R|^2} \varepsilon= \mathcal{N}(\varepsilon)- \Psi- \sum_{j=1}^2 \frac{1}{\lambda_j^2} S_j(t,x) e^{i\gamma_j+ i\beta_j \cdot x},
\end{equation} 
where
\begin{equation}
    \mathcal{N}(\varepsilon)= 2\phi_{\mathrm{Re} (\varepsilon \overline{R})} \varepsilon+ \phi_{|\varepsilon|^2}R+ \phi_{|\varepsilon|^2} \varepsilon.
\end{equation}

\;\\
\textbf{Step 2:} We estimate the closeness between $g$ and $g^{(N)}$ in terms of $\|\varepsilon\|_{H^1}$. 

Define the modulation error 
\begin{equation} \begin{aligned}
    Mod(t)= \sum_{j=1}^2 \Bigg( &\Big| \dot{\alpha}_j(t)- 2\beta_j(t) \Big| + \Big| \dot{\beta}_j(t)- B_j(P(t)) \Big| + \Big| \dot{\lambda}_j(t)- M_j(P(t)) \Big| \\
    &+ \bigg| \dot{\gamma}_j(t)- \frac{1}{\lambda_j^2(t)}+ |\beta_j(t)|^2+ \dot{\beta}_j(t) \cdot \alpha_j(t)- \frac{\| Q \|_{L^2}^2}{4\pi \lambda_{j+1}} \frac{1}{|\alpha|} \bigg| \Bigg). 
\end{aligned} \end{equation}
Let $\theta_j= g_j \theta$, where $\theta(t,x)$ satisfies
\begin{equation} \label{eq decay of theta}
    |\theta(t,x)| \le C_0 e^{-c_0|x|}, \quad \forall t>0,\ x \in \mathbb{R}^3.
\end{equation}

By \eqref{eq equation of epsilon}, we can compute
\begin{equation} \label{eq modulation equation with theta} \begin{aligned}
    \frac{\d}{\dt} \Im \int \varepsilon \ol{\theta_j}= &\ \Re \int \varepsilon \left( \ol{i\partial_t \theta_j+ \Delta \theta_j- 2\phi_{\Re (\theta_j \ol{R})}R- \phi_{|R|^2} \theta_j} \right) \\
    &+ \Re \int \Big( \Psi- \mathcal{N}(\varepsilon) \Big) \ol{\theta_j}+ \sum_{k=1}^2 \Re \int \frac{1}{\lambda_k^2} S_k(t,x) e^{i\gamma_k+ i\beta_k \cdot x} \ol{\theta_k}.
\end{aligned} \end{equation}
By \eqref{eq estimate of B_j, M_j} and \eqref{eq decay of theta}, we have
\begin{equation}
    i\partial_t \theta_j+ \Delta \theta_j= \frac{1}{\lambda_j^4} \left( i\lambda_j^2 \partial_t \theta+ \Delta \theta- \theta \right) e^{i\gamma_j+ i\beta_j \cdot x}+ O_0 \left( \left( r_0^{-1}+ Mod \right) e^{-c_0 |x-\alpha_j|} \right).
\end{equation}
By the localization of $V_j$, \eqref{eq decay of theta} and the proof of Proposition \ref{prop accuracy of approximate solution}, we have
\begin{align}
    \phi_{\Re (\theta_j \overline{R})} R &= \frac{1}{\lambda_j^4} \phi_{\Re (\theta \ol{V_j})} V_j e^{i\gamma_j+ i\beta_j \cdot x}+ O_0 \left( \big( r_0^{-1}+ e^{-c_0 r_0} \big) e^{-c_0 |x-\alpha_j|} \right)
\end{align}
and
\begin{equation}
    \phi_{|R|^2} \theta_j= \frac{1}{\lambda_j^4} \phi_{|V_j|^2} \theta e^{i\gamma_j+ i\beta_j \cdot x}+ O_0 \left( \big( r_0^{-1}+ e^{-c_0 r_0} \big) e^{-c_0 |x-\alpha_j|} \right).
\end{equation}

Collecting the terms of degree $1$ in $\theta$, we define
\begin{equation}
    L_j \theta:= -\Delta \theta+ \theta+ 2\phi_{\mathrm{Re} (\theta \overline{V_j})} V_j+ \phi_{|V_j|^2} \theta.
\end{equation}
This is the linearized operator around the profile $V_j$, and we have
\begin{equation} \begin{aligned}
    &i\partial_t \theta_j+ \Delta \theta_j- 2\phi_{\mathrm{Re} (\theta_j \overline{R})}R- \phi_{|R|^2} \theta_j \\
    =&\ \frac{1}{\lambda_j^4} \left( i\lambda_j^2 \partial_t \theta- L_j \theta \right) e^{i\gamma_j+ i\beta_j \cdot x}+ O_0 \Big( \big( r_0^{-1}+ Mod \big) \max_k e^{-c_0 |x-\alpha_k|} \Big).
\end{aligned} \end{equation}
Thus, from \eqref{eq modulation equation with theta}, \eqref{eq estimate of Psi} and \eqref{eq decay of theta}, we get
\begin{equation} \begin{aligned}
    \frac{\d}{\dt} \mathrm{Im} \int \varepsilon \overline{\theta_j}= &\ \mathrm{Re} \int \frac{\varepsilon}{\lambda_j^4} \overline{\left( i\lambda_j^2 \partial_t \theta- L_j \theta \right) e^{i\gamma_j+ i\beta_j \cdot x}}+ \frac{1}{\lambda_j^6} \Re \int S_j \overline{\theta}\\
    &+ O_0 \left( \Big( r_0^{-1}+ Mod \Big) \| \varepsilon \|_{H^1}+ e^{-c_0 r_0}+ \| \varepsilon \|_{H^1}^2 \right). 
\end{aligned} \end{equation}

We will take $\theta$ to be $iV_j$, $\nabla V_j$, $\Lambda V_j$ and $iy_j V_j$, which correspond to $\gamma_j$, $\alpha_j$, $\lambda_j$ and $\beta_j$, respectively. By \eqref{eq orthogonality}, the left-hand side always vanishes. By \eqref{eq estimate of V_j} and \eqref{eq estimate of B_j, M_j}, we always have
\begin{equation} \label{eq estimate on theta: partial_t}
    \partial_t \theta= O_0 \left( \left( r_0^{-1}+ Mod \right) e^{-c_0 |x-\alpha_j|} \right).
\end{equation}
By \eqref{eq estimate of V_j} and \eqref{eq ground state}, we know 
\begin{equation}
    W_j:= -\Delta V_j+ V_j+ \phi_{|V_j|^2} V_j
\end{equation}
satisfies $|W_j| \le C_0 r_0^{-1} e^{-c_0 |x-\alpha_j|}$. Direct computation yields
\begin{equation} \begin{gathered}
    L_j (iV_j) = iW_j, \quad L_j (\Lambda V_j) = (\Lambda+2) W_j- 2 V_j, \\
    L_j (\nabla V_j) = \nabla W_j, \quad L_j (iy_j V_j) =iy_j W_j- 2i \nabla V_j.
\end{gathered} \end{equation}
By \eqref{eq orthogonality}, we always have
\begin{equation}
    L_j \theta= f+ O_0 \left( r_0^{-1} e^{-c_0 |x-\alpha_j|} \right),
\end{equation}
where $f$ is a function such that $\mathrm{Re} \int \varepsilon (\ol{g_j f})= 0$. Thus
\begin{equation} \label{eq estimate on theta: L_j}
    \Re \int \frac{\varepsilon}{\lambda_j^4} \ol{ L_j \theta\ e^{i\gamma_j+ i\beta_j \cdot x}}= O_0 \left( r_0^{-1} \| \varepsilon \|_{H^1} \right).
\end{equation}
Finally, using \eqref{eq definition of S_j^N} and \eqref{eq estimate of V_j} and summing over the four choices of $\theta$, we have
\begin{equation} \label{eq estimate on theta: S_j}
    \sum_{\theta} \sum_{j=1}^2 \left| \mathrm{Re} \int S_j \overline{\theta} \right| \ge c_0 Mod- C_0 r_0^{-1} Mod.
\end{equation}

Therefore, combining \eqref{eq estimate on theta: partial_t}, \eqref{eq estimate on theta: L_j}, \eqref{eq estimate on theta: S_j} and taking $r_0$ large enough, we obtain
\begin{equation} \label{eq estimate of Mod with epsilon}
    Mod(t) \le C_0 r_0^{-1} \| \varepsilon \|_{H^1}+ e^{-c_0 r_0}+ C_0 \| \varepsilon \|_{H^1}^2. 
\end{equation}

\;\\
\textbf{Step 3:} We give an upper bound for $\| \varepsilon \|_{H^1}$. 

Let $\varphi_1, \varphi_2$ be smooth cutoff functions such that
\begin{equation} \label{eq cutoff} \begin{gathered}
    \varphi_j(t,x) \ge 0, \quad \sum_{j=1}^2 \varphi_j(t,x) \equiv 1, \quad |\partial_t \varphi_j|+ |\nabla \varphi_j| \le C_0 r^{-1}, \\
    \varphi_j(t,x)= \left\{ \begin{aligned}
        &1, \quad |x-\alpha_j(t)| \le c_0 r(t), \\
        &0, \quad |x-\alpha_{j+1}(t)| \le c_0 r(t).
    \end{aligned} \right.
\end{gathered} \end{equation} 
By \eqref{eq estimate of V_j}, we have
\begin{equation} \label{eq localization of R_j and cutoff}
    |\varphi_j R- R_j| \le C_0 e^{-c_0 r(t)}, \quad \forall t>0,\ x \in \mathbb{R}^3.
\end{equation}

Define $\G(\varepsilon)= \G_1+\G_2+\G_3$, where
\begin{gather*}
    \G_1= \int |\nabla \varepsilon|^2+ \int \phi_{|R|^2} |\varepsilon|^2- 2\int |\nabla \phi_{\Re (\varepsilon \ol{R}) }|^2+ 2\int \phi_{\Re (\varepsilon \ol{R}) } |\varepsilon|^2- \frac{1}{2} \int |\nabla \phi_{|\varepsilon|^2} |^2, \\
    \G_2= \sum_{j=1}^2 \Big( \frac{1}{\lambda_j^2}+ |\beta_j|^2 \Big) \int \varphi_j |\varepsilon|^2, \qquad \G_3= -2\sum_{j=1}^2 \beta_j \int \varphi_j \Im (\nabla \varepsilon \ol{\varepsilon}).
\end{gather*}
By \cite[Proposition~4.1]{Hartree3Dmultisoliton}, for $r_0$ large enough, we have $\G(\varepsilon) \ge c_0 \| \varepsilon \|_{H^1}^2$. 

Then we compute
\begin{equation} \begin{aligned}
    \frac{\d \G_1}{\dt}= &\ -2\Im \int i\partial_t \overline{\varepsilon} \left( \Delta \varepsilon- \phi_{|R|^2} \varepsilon- 2\phi_{\mathrm{Re} (\varepsilon \ol{R})} R- \mathcal{N}(\varepsilon) \right) \\
    &\ +4\Re \int \phi_{\Re (\varepsilon \ol{R})} \varepsilon \partial_t \ol{R}+ 2\int \phi_{\Re (\partial_t R \ol{R})} |\varepsilon|^2+ 2\Re \int \phi_{|\varepsilon|^2} \varepsilon \partial_t \ol{R}.
\end{aligned} \end{equation}
By \eqref{eq equation of epsilon} and \eqref{eq estimate of Psi}, the first line is $O_0 \left( \big( e^{-c_0 r_0}+ Mod \big) \| \varepsilon \|_{H^1} \right)$. 

For the second line, using \eqref{eq estimate of B_j, M_j}, we have
\begin{equation} \begin{aligned}
    \partial_t R_j &= \frac{1}{\lambda_j^2} \left( -\dot{\alpha}_j \cdot \nabla V_j- i(\dot{\gamma}_j- \dot{\beta}_j \cdot \alpha_j) V_j \right) e^{i\gamma_j+ i\beta_j \cdot x}+ O_0 \left( r_0^{-1} e^{-c_0 |x-\alpha_j|} \right) \\
    &= -2\beta_j \cdot \nabla R_j+ i\big( \frac{1}{\lambda_j^2}+ |\beta_j|^2 \big) R_j+ O_0 \left( \left( r_0^{-1}+ Mod \right) e^{-c_0 |x-\alpha_j|} \right).
\end{aligned} \end{equation}
Combining this with \eqref{eq estimate of Mod with epsilon}, we get
\begin{equation} \label{eq estimate of G_1} \begin{aligned}
    \frac{\d \G_1}{\dt}= &\sum_{j=1}^2 \Bigg( 4\Big(  \frac{1}{\lambda_j^2}+ |\beta_j|^2 \Big) \int \phi_{\mathrm{Re}( \varepsilon \overline{R})} \mathrm{Im} (\varepsilon \overline{R_j})- 8\int \phi_{\mathrm{Re} (\varepsilon \overline{R})} \mathrm{Re} (\varepsilon \beta_j \cdot \nabla \overline{R_j}) \\
    &\qquad \quad- 4\int \phi_{\mathrm{Re} (\beta_j \cdot \nabla R_j \overline{R_j})} |\varepsilon|^2 \Bigg) \\
    &+ O_0 \left( e^{-c_0 r_0} \| \varepsilon \|_{H^1}+ r_0^{-1} \| \varepsilon \|_{H^1}^2+ \| \varepsilon \|_{H^1}^3 \right).
\end{aligned} \end{equation} 

Using \eqref{eq cutoff}, we have
\begin{equation}
    \frac{\d \G_2}{\dt}= \sum_{j=1}^2 2\Big( \frac{1}{\lambda_j^2}+ |\beta_j|^2 \Big) \int \varphi_j \mathrm{Im} (i\partial_t \varepsilon \overline{\varepsilon})+ O_0 \left( r_0^{-1} \| \varepsilon \|_{H^1}^2 \right). 
\end{equation}
Then by \eqref{eq equation of epsilon} and \eqref{eq estimate of Psi}, we have
\begin{equation} \begin{aligned}
    \frac{\d \G_2}{\dt} &= \sum_{j=1}^2 2\Big( \frac{1}{\lambda_j^2}+ |\beta_j|^2 \Big) \int \varphi_j \mathrm{Im} (2\phi_{\mathrm{Re}(\varepsilon \overline{R})} R \overline{\varepsilon}) \\
    &\quad + O_0 \left( e^{-c_0 r_0} \| \varepsilon \|_{H^1}+ Mod \| \varepsilon \|_{H^1}+ r_0^{-1} \| \varepsilon \|_{H^1}^2+ \| \varepsilon \|_{H^1}^3 \right).
\end{aligned} \end{equation}
Finally, using \eqref{eq estimate of Mod with epsilon} and \eqref{eq localization of R_j and cutoff}, we get
\begin{equation} \label{eq estimate of G_2} \begin{aligned}
    \frac{\d \G_2}{\dt} &= -\sum_{j=1}^2 4\Big( \frac{1}{\lambda_j^2}+ |\beta_j|^2 \Big) \int \phi_{\mathrm{Re}(\varepsilon \overline{R})} \mathrm{Im} (\varepsilon \overline{R_j}) \\
    &\quad + O_0 \left( e^{-c_0 r_0} \| \varepsilon \|_{H^1}+ r_0^{-1} \| \varepsilon \|_{H^1}^2+ \| \varepsilon \|_{H^1}^3 \right).
\end{aligned} \end{equation} 

Similarly, we can compute
\begin{equation} \label{eq estimate of G_3} \begin{aligned}
    \frac{\d \G_3}{\dt} &= -\sum_{j=1}^2 4\beta_j \int \varphi_j \mathrm{Re}(i\partial_t \varepsilon \nabla \overline{\varepsilon})+ O_0 \left( r_0^{-1} \| \varepsilon \|_{H^1}^2 \right) \\
    &= -\sum_{j=1}^2 4\beta_j \int \varphi_j \mathrm{Re} \left( 2\phi_{\mathrm{Re} (\varepsilon \overline{R})} R \nabla \overline{\varepsilon}+ \phi_{|R|^2} \varepsilon \nabla \overline{\varepsilon} \right) \\
    &\quad + O_0 \left( e^{-c_0 r_0} \| \varepsilon \|_{H^1}+ Mod \| \varepsilon \|_{H^1}+ r_0^{-1} \| \varepsilon \|_{H^1}^2+ \| \varepsilon \|_{H^1}^3 \right) \\
    &= \sum_{j=1}^2 \left( 8\int \phi_{\mathrm{Re} (\varepsilon \overline{R})} \mathrm{Re}(\varepsilon \beta_j \cdot \nabla \overline{R_j})+ 4\int \phi_{\mathrm{Re} (\beta_j \cdot \nabla R_j \overline{R_j})} |\varepsilon|^2 \right) \\
    &\quad + O_0 \left( e^{-c_0 r_0} \| \varepsilon \|_{H^1}+ r_0^{-1} \| \varepsilon \|_{H^1}^2+ \| \varepsilon \|_{H^1}^3 \right).
\end{aligned} \end{equation}

Combining \eqref{eq estimate of G_1}, \eqref{eq estimate of G_2}, \eqref{eq estimate of G_3}, we deduce
\begin{equation}
    \left| \frac{\d}{\dt} \G(\varepsilon(t)) \right| \le C_0 r_0^{-1} \| \varepsilon \|_{H^1}^2+ C_0 e^{-c_0 r_0} \| \varepsilon \|_{H^1}+ C_0 \| \varepsilon \|_{H^1}^3.
\end{equation}
By the coercivity of $\G$, we have
\begin{equation}
    \left| \frac{\d \G}{\dt}  \right| \le C_0 r_0^{-1} \G+ C_0 e^{-c_0 r_0} \G^{\frac{1}{2}}+ C_0 \G^{\frac{3}{2}}.
\end{equation}
Then using Gr\"onwall's inequality, we deduce 
\begin{equation}
    \G(t) \le C_0 \sup_{\tau \in [0,t]} \left( e^{-c_0 r_0} \G^{\frac{1}{2}}(\tau)+ C_0 \G^{\frac{3}{2}}(\tau) \right) e^{\frac{C_0 t}{r_0}}+ C_0 \kappa^2 e^{\frac{C_0 t}{r_0}}.
\end{equation}
Since $t \le \nu r_0^2$, by taking $\nu$ small enough, then $r_0$ large enough, and finally $\kappa$ small enough, we obtain $\G(\varepsilon(t)) \le C_0 e^{-c_0 r_0}$ by bootstrap. Using the coercivity of $\G$, this yields 
\begin{equation}
    \| \varepsilon \|_{H^1} \le C_0 e^{-c_0 r_0}.
\end{equation}

\;\\
\textbf{Step 4:} We give an upper bound for $|g- g^{(N)}|$ independent of $\|\varepsilon\|_{H^1}$.

Now \eqref{eq estimate of Mod with epsilon} implies $Mod(t) \le C_0 e^{-c_0 r_0}$. Recall $\lambda=(\lambda_1, \lambda_2)$. Let 
\begin{equation}
    F(t)= r_0^{-2} \big| \alpha(t)- \alpha^{(N)}(t) \big|^2+ \big| \beta(t)- \beta^{(N)}(t) \big|^2+ \big| \lambda(t)- \lambda^{(N)}(t) \big|^2.
\end{equation}
Then
\begin{equation}
    \dot{F} \le C_0 F^{\frac{1}{2}} \Big( r_0^{-1} \big| \dot{\alpha}- \dot{\alpha}^{(N)} \big|+ \big| \dot{\beta}- \dot{\beta}^{(N)} \big|+ \big| \dot{\lambda}- \dot{\lambda}^{(N)} \big| \Big).
\end{equation}
Since $P^{(N)}$ satisfies \eqref{eq 2-body N-th} and by the definition of $Mod$, we get
\begin{equation}
    \dot{F} \le C_0 F^{\frac{1}{2}} \bigg( Mod+ r_0^{-1} \big| \beta- \beta^{(N)} \big|+ \sum_{j=1}^2 \Big[ \big| B_j(P)- B_j(P^{(N)}) \big|+ \big| M_j(P)- M_j(P^{(N)}) \big| \Big] \bigg).
\end{equation}
Using \eqref{eq estimate of B_j, M_j} and the fundamental theorem of calculus, we have
\begin{equation}
    \big| B_j(P)- B_j(P^{(N)}) \big|+ \big| M_j(P)- M_j(P^{(N)}) \big| \le C_0 r_0^{-1} F^{\frac{1}{2}}.
\end{equation}
We thus obtain
\begin{equation}
    \dot{F} \le C_0 r_0^{-1} F+ C_0 e^{-c_0 r_0} F^{\frac{1}{2}}.
\end{equation}
Then we deduce $F(t) \le C_0 e^{-c_0 r_0}$ for $t \le \nu r_0^2$ by Gr\"onwall's inequality as above. This gives the upper bound for $|P- P^{(N)}|$. The estimate for $\gamma$ follows by direct integration.

\;

In conclusion, in Step 3, we obtain 
\begin{equation}
    \| u(t, \cdot)- R_g^{(N)}(t, \cdot) \|_{H^1} \le C_0 e^{-c_0 r_0}, \quad \forall 0 \le t \le \nu r_0^2
\end{equation}
and in Step 4, we obtain
\begin{equation}
    |g- g^{(N)}| \le C_0 e^{-c_0 r_0}, \quad \forall 0 \le t \le \nu r_0^2.
\end{equation}
Thus we deduce
\begin{equation}
    \| u(t, \cdot)- R_{g^{(N)}}^{(N)}(t, \cdot) \|_{H^1} \le C_0 e^{-c_0 r_0}, \quad \forall 0 \le t \le \nu r_0^2.
\end{equation}
By \eqref{eq estimate of V_j}, we have 
\begin{equation}
    \| R_{g^{(N)}}^{(N)}(t, \cdot)- (g_1^{(N)}Q) (t,\cdot)- (g_2^{(N)}Q) (t,\cdot) \|_{H^1} \le C_0 r_0^{-3}.
\end{equation}
Therefore, we conclude \eqref{eq finite time stability} by taking $r_0$ large enough so that $C_0 e^{-c_0 r_0}+ C_0 r_0^{-3}< \delta$.
\end{proof}

\section*{Acknowledgments}
I would like to thank Professor Wilhelm Schlag for helpful discussions.

\appendix

\section{Growth estimates for the recursive sequence} \label{appendix sequence lemma}

We prove the following lemma, which completes the proof of Lemma \ref{lem quantitative description of b_j m_j T_j}.

\begin{lem}
Given $C_0, a_0>0$, define the sequence $\{a_n\}_{n=1}^\infty$ by
\begin{equation}
    a_n= C_0 \sum_{\substack{k+l+m \le n \\ 0 \le k,l,m \le n-1}} a_k a_l a_m, \quad \forall n \ge 1.
\end{equation}
Then there exists a constant $C>0$ depending on $C_0, a_0$ such that $a_n \le C^n$ for all $n \ge 1$.
\end{lem}

\begin{proof}
We may assume $a_0=1$ by considering the sequence $a_n/a_0$ instead.

Define the formal power series
\begin{equation}
    g(x)= \sum_{n=0}^\infty a_n x^n.
\end{equation}
Let $[x^n]$ denote the coefficient of $x^n$ in a formal power series. Then
\begin{equation}
    [x^n]\frac{g^3(x)}{1-x}= \sum_{\substack{k+l+m \le n \\ 0 \le k,l,m \le n}} a_k a_l a_m= C_0^{-1} a_n+ 3a_n, \quad \forall n \ge 1.
\end{equation}
Note that $[x^0]\frac{g^3(x)}{1-x}=1$. We thus obtain a formal equation
\begin{equation}
    \frac{g^3(x)}{1-x}- 1= (C_0^{-1}+ 3)(g(x)- 1). 
\end{equation}
Let $f(x)=g(x)-1$. Then in the formal power series sense, we have
\begin{equation}
    f(x)= xG(f(x)),
\end{equation}
where
\begin{equation}
    G(f)= \frac{1+(3+C_0^{-1})f}{C_0^{-1}- 3f- f^2}.
\end{equation}
Since $G(0) \neq 0$, by the Lagrange inversion theorem \cite[Theorem~5.4.2]{Lagrangeinversion}, we deduce
\begin{equation}
    n[x^n]f(x)= [x^{n-1}] G(x)^n, \quad \forall n \ge 1.
\end{equation}
Thus $a_n= \frac{1}{n} [x^{n-1}] G(x)^n$ for $n \ge 1$. It suffices to show $[x^{n-1}] G(x)^n \le C^n$. 

Let us assume $C_0=1$ for simplicity. The general case can be handled similarly, but the formula is less transparent. Let $u(x)= (1-3x-x^2)^{-n}$ and $v(x)= (1+4x)^n$. Then
\begin{equation}
    [x^{n-1}]G(x)^n= \sum_{k=0}^{n-1} [x^k] u(x) \cdot [x^{n-k-1}]v(x).
\end{equation}
We have $[x^{n-k-1}]v(x) \le v(1)= 5^n$. For $u$, expanding $(1-y)^{-n}$ with $y=3x+x^2$ gives
\begin{equation}
    u(x)= \sum_{j=0}^\infty \binom{n+j-1}{j} (3x+x^2)^j. 
\end{equation}
Then we deduce
\begin{align}
    [x^k] u(x) &= [x^k] \sum_{j=0}^k \binom{n+j-1}{j} (3x+x^2)^j \\
    &\le \sum_{j=0}^k \binom{n+j-1}{j} (3x+x^2)^j \bigg|_{x=1} \le 4^k \sum_{j=0}^k \binom{n+j-1}{j}= 4^k \binom{n+k}{k}.
\end{align}
Since $\binom{n+k}{k}$ is increasing in $k$ and $k \le n-1$, we have
\begin{equation}
    [x^k] u(x) \le 4^n \binom{2n}{n} \le 4^n \cdot 2^{2n}= 16^n.
\end{equation}
We then conclude that $[x^{n-1}]G(x)^n \le 80^n$. This completes the proof.
\end{proof}

\bibliographystyle{amsplain}
\bibliography{ref}

\end{document}